\documentclass[14pt]{article}
\usepackage{amsmath}
\usepackage{amssymb}
\usepackage{mathrsfs}
\usepackage[all,cmtip]{xy}
\usepackage{amsthm,color}
\usepackage{tikz-cd}
\usepackage{diagrams}
\usepackage{authblk}

\let\margin\marginpar
\newcommand\myMargin[1]{\margin{\raggedright\scriptsize #1}}
\renewcommand{\marginpar}[1]{\myMargin{#1}}

\newtheorem{lemma}{Lemma}[section]
\newtheorem{theorem}[lemma]{Theorem}
\newtheorem{corollary}[lemma]{Corollary}

\newtheorem{prop}[lemma]{Proposition}
\newtheorem*{theorem*}{Main Theorem}

\theoremstyle{definition}
\newtheorem{definition}[lemma]{Definition}

\newtheorem{remark}[lemma]{Remark}

\theoremstyle{remark}
\newtheorem*{proof*}{Proof}
\numberwithin{equation}{section}

\newcommand{\Cone}{\operatorname{Cone}}

\newcommand{\SL}{\operatorname{SL}}
\newcommand{\sub}{\subset}
\newcommand{\lan}{\langle}
\newcommand{\ran}{\rangle}
\newcommand{\coker}{\operatorname{coker}}
\newcommand{\Aut}{\operatorname{Aut}}

\def\gr{{\mathbf{gr}}}
\def\D{\mathrm{D}}
\def\Pic{\mathrm{Pic}}
\def\coh{\mathrm{coh}}

\def\Ext{{\mathrm{Ext}}}

\def\Hom{{\mathrm{Hom}}}
\def\sHom{{\mathscr{H}om}}

\def\sEnd{{\mathscr{E}nd}}

\newcommand{\Om}{\Omega}
\newcommand{\om}{\omega}
\newcommand{\ov}{\overline}
\def\Perf{{\underline{Perf}}}

\def\Ker{{\text{Ker}}}
\def\Cok{{\text{Cok}}}

\def\Id{{\mathrm{I}}}
\def\p{{\prime}}
\def\pp{{\prime\prime}}
\def\pr{{\mathrm{pr}}}

\def\dVt{\RR\underline{Vect}}

\def\dCx{\RR{\underline{Cplx}}}

\def\PP{{\mathbb P}}

\def\ZZ{{\mathbb Z}}
\def\CC{{\mathbb C}}
\def\LL{{\mathbb L}}
\def\TT{{\mathbb T}}

\def\RR{{\mathbb R}}
\def\HH{{\mathbb H}}
\def\GG{{\mathbb G}}
\def\AA{{\mathbb A}}

\def\cF{{\cal{F}}}
\def\cE{{\cal{E}}}
\def\cO{{\cal{O}}}
\def\cC{{\cal{C}}}

\def\cM{{\cal{M}}}

\def\cA{A}
\def\cB{{\cal{B}}}
\def\cH{{\cal{H}}}
\def\cP{{\cal{P}}}
\def\cS{{\cal{S}}}
\def\cT{{\cal{T}}}
\def\cU{{\cal{U}}}
\def\cV{{\cal{V}}}

\def\cG{{\cal{G}}}

\def\cX{{\cal{X}}}
\def\cY{{\cal{Y}}}

\def\fs{{\mathfrak{s}}}
\def\ft{{\mathfrak{t}}}
\def\fg{{\mathfrak{g}}}

\def\fm{{\mathfrak{m}}}

\def\tr{{\mathrm{tr}}}

\def\rk{{\mathrm{rk}}}

\def\adj{{\mathrm{adj}}}
\def\Res{{\mathrm{Res}}}
\def\deg{{\mathrm{deg}}}
\def\ch{{\mathrm{ch}}}

\def\ker{{\mathrm{ker}}}

\def\mod{{~\mathrm{mod}~}}

\def\ep{{\epsilon}}
\def\gl{{\mathfrak{gl}}}

\def \ot {\otimes}

\def\ol{\overline}

\def\wt{\widetilde}
\def\gr{\mathrm{gr}}
\def\sn{\mathrm{sn}}
\def\dn{\mathrm{dn}}
\def\cn{\mathrm{cn}}

\def\reg{\mathrm{reg}}

\newcommand{\Trp}{{\mathcal Trp}}
\newcommand{\TS}{{\mathcal TS}}
\newcommand{\divis}{\operatorname{div}}
\newcommand{\dis}{\operatorname{dis}}

\def \la {\langle}

\makeatletter
\DeclareRobustCommand\frownotimes{\mathbin{\mathpalette\frown@otimes\relax}}
\newcommand{\frown@otimes}[2]{%
  \vbox{
    \ialign{##\cr
      \hidewidth$\m@th#1{}_\frown$\kern-\scriptspace\hidewidth\cr
      \noalign{\nointerlineskip\kern-1pt}
      $\m@th#1\otimes$\cr
    }%
  }%
}

\title{Shifted Poisson geometry and meromorphic matrix algebras over an elliptic curve}
\date{}
\author[1]{Zheng Hua\thanks{huazheng@maths.hku.hk}}
\author[2]{Alexander Polishchuk \thanks{apolish@uoregon.edu}}
\affil[1]{Department of Mathematics, the University of Hong Kong, Hong Kong SAR, China}
\affil[2]{University of Oregon and National Research University Higher School of Economics}

\begin{document}
\maketitle

\begin{center}
\begin{abstract}
\vspace{1cm}
In this paper we classify symplectic leaves of the regular part of the projectivization of the space of meromorphic endomorphisms of a stable vector bundle on an elliptic curve, using the study of shifted Poisson structures on the moduli of complexes from our previous work \cite{HP17}. This Poisson ind-scheme is closely related to the ind Poisson-Lie group associated to Belavin's elliptic $r$-matrix, studied by Sklyanin, Cherednik and Reyman and Semenov-Tian-Shansky. 
Our result leads to a classification of symplectic leaves on the regular part of meromorphic matrix algebras over an elliptic curve, which can be viewed as the Lie algebra of the above-mentioned ind Poisson-Lie group.
We also describe the decomposition of the product of leaves under the multiplication morphism and show the invariance of Poisson structures under autoequivalences of the derived category of coherent sheaves on an elliptic curve. 

\end{abstract} 
\end{center}
\newpage
\section{Introduction}

This paper is a continuation of \cite{HP17}. Recall that in \cite{HP17} we constructed a natural $0$-shifted Poisson structure on the
derived moduli stack $\dCx(C)$ of complexes of vector bundles
over an elliptic curve $C$. Furthermore, we showed that homotopy fibers of the natural morphism from $\dCx(C)$ to the product
of the moduli stack of perfect complexes and the moduli stack of graded vector bundles, have a $0$-shifted symplectic structure,
so they can be viewed as {\it derived symplectic leaves} of our $0$-shifted Poisson structure.

In the present paper we proceed a little further in the special case of complexes of the form $E\to E(D)$, where
$E$ is a stable vector bundle on $C$, and $D$ is a divisor on $C$ of positive degree
(both $E$ and $D$ are fixed). The coarse moduli space of such complexes, denoted by $M(E,D)$, is the projectivization of the linear space $A(E,D):=\Hom(E,E(D))$. 
In the context of integrable system and quantum algebra, the Poisson structure on $A(E,D)$ was first studied by Sklyanin \cite{Skl82} and Cherednik \cite{Ch1}\cite{Ch2}. In the case when $E$ has rank $2$ and $D$ has degree $1$, Sklyanin constructed its quantization in \cite{Skl82}. The quantization of $A(E,D)$ for general $E$ and $D$ was constructed by Cherednik (see \cite{Ch2}). The corresponding quantum algebras are known as \emph{generalized Sklyanin elliptic algebras}. Note that the orginal construction of Sklyanin (and of Cherednik) is based on Belavin's elliptic solution of the Yang-Baxter equation. The relation of the corresponding Poisson bracket to vector bundles on elliptic curves was clarified by Hurtubise and Markman (see \cite{HM02}). Our geometric construction of a Poisson structure on $A(E,D)$
differs from that of Hurtubise and Markman: we first construct a Poisson structure on its projectivization $M(E,D)$ and then 
characterize its Poisson lift to $A(E,D)$
as a unique one for which the determinant morphism is a Casimir map (see Section \ref{sec:det}).

We also consider the inductive limit $A(E)$ (resp., $M(E)$) of $A(E,D)$ (resp., $M(E,D)$) over all effective divisors $D$. 
This is an infinite dimensional Poisson ind-scheme. 
It can be viewed as the meromorphic part of the Manin triple associated to a formal loop Lie algebra with the Poisson structure defined by Belavin's elliptic $r$-matrix (see Section 12 of \cite{RSTS}). We show that product map on $A(E)$ is Poisson.

The main result of this article is a complete description of symplectic leaves on the regular part $M^{\reg}\subset M(E,D)$, i.e., the open subset consisting of injective endomorphisms modulo isomorphisms of complexes, as well as a classification of symplectic leaves on the regular part of $A(E,D)$.

For a partition $\nu$ of $r\cdot\deg(D)$ we denote by $S_\nu C$ the stratum in the symmetric power of $C$ consisting
of divisors of the form $\sum \nu_i x_i$, and we denote by $S_\nu^{rD} C\sub S_\nu C$ the fiber of the Abel-Jacobi map
$S_\nu C\to \Pic(C)$ over $\cO(rD)$.

\begin{theorem*} Let $E$ be a stable vector bundle of rank $r$ and degree $d$, and $D$ be an effective divisor of degree $k$ on a complex elliptic curve $C$. Given a partition $\nu=(\nu_1\geq\nu_2\geq\ldots\geq \nu_n)$ of $r\cdot k$, and a collection of partitions $\Lambda_\nu=\{\lambda^i\}_{i=1}^n$ such that $|\lambda^i|=\nu_i$, 
we set $l_{\max}(\Lambda_\nu):=\max\{l(\lambda^i)|i=1,\ldots n\}$. 
\begin{enumerate}
\item[(1)]
The set of topological types of the symplectic leaves of $M^{\reg}$ is in one-to-one correspondence with the set of pairs $(\nu,\Lambda_\nu)$ such that $l_{\max}(\Lambda_\nu)\leq r$.
\item[(2)] Given a pair $(\nu,\Lambda_\nu)$ such that $l_{\max}(\Lambda_\nu)\leq r$, the union 
of symplectic leaves of $M^\reg$ of the topological type $(\nu,\Lambda_\nu)$ fibers smoothly over
$S^{rD}_\nu C$.
\item[(3)] Let $A^{\reg}\subset A(E,D)$ be the open subset consisting of injective endomorphisms and 
let $S\sub A^{\reg}$ be a nonempty fiber of the determinant map $\det:A(E,D)\to H^0(C,\cO(rD))$. Let us consider
the quotient map $\pi:S\to M^\reg$, and let $F\sub M^\reg$ be a symplectic leaf. Then
connected components of $\pi^{-1}(F)$ are symplectic leaves in $A^\reg$ (and all leaves are obtained by this construction).
\end{enumerate}
\end{theorem*}
The above theorem is proved in Section \ref{sec:symleaves}. In particular, part $(1)$ corresponds to Theorem \ref{lem_classleaf1}, part $(2)$ corresponds to Theorem \ref{fibrationleaf}, and part $(3)$ corresponds to Corollary \ref{leafofA}.
We also show that the product of two symplectic leaves under the monoid structure on $M(E)$ 
is the union of open subsets in a finite collection of symplectic leaves that can be explicitly described 
(see Theorem \ref{multleaf}).

Note that on finite dimensional complex reductive Lie groups with Poisson-Lie structures defined by quasi-triangular $r$-matrices, the symplectic leaves are orbits of the dressing action (see \cite{Ya02}). In the infinite dimensional case, this method needs additional careful treatement (see \cite{Williams13} for a discussion for the trigonometric case). In this  paper, we take a different approach via derived algebraic geometry, where the action of infinite dimensional (ind)-group can be avoided. 

 In a somehow different direction, we prove the invariance of Poisson structure on the moduli space of complexes under the auto-equivalence of the derived category of coherent sheaves on $C$ (Theorem \ref{FM-thm}). As a corollary, we set up a Poisson isomorphism between $M(E,D)$ and the Poisson moduli space constructed by Feigin and Odesskii \cite{FO95}.

The paper is organized as follows. In Section \ref{sec:loopalg}, we recall the construction of shifted Poisson structure on moduli space of complexes and introduce the space $M(E,D)$ as a special case. Section \ref{sec:symleaves} is the main body of this article. First, we prove that the coarse moduli spaces of the derived symplectic leaves are smooth schemes and show that the classical shadow of the 0-shifted symplectic structure descends to a symplectic structure on the coarse moduli spaces. Then the classification results are proved. In Section \ref{sec:rank-2}, we classify all symplectic leaves for rank 2 case without assuming the regular condition, and give some concrete examples of leaves. In Section \ref{sec:Rel-FO}, we prove the invariance of the Poisson structure under auto-equivalences. 

\paragraph{Acknowledgments.} We are grateful to Jiang-Hua Lu and Yongchang Zhu for many valuable discussions. The second
author thanks Institut Mathematique Jussieu and Institut des Hautes Etudes Scientifiques for hospitality and excellent
working conditions during preparation of this paper.
The research of Z.H. is supported by RGC general research fund no. 17330316  and NSFC Science Fund for Young Scholars no. 11401501. The research of
A.P. is supported in part by the NSF grant DMS-1700642 and by the Russian Academic Excellence Project `5-100'.

\paragraph{Notations:} We work over the field of complex numbers. 
We use letters $\cX,\cY,\ldots$ to denote the stacks (derived or un-derived), and straight letters $X,Y,\ldots$ to denote their underlying coarse moduli schemes (if exist). The tangent and cotangent complexes of $\cX$ are denoted by $\TT_\cX$ and $\LL_\cX$. There is a truncation functor from the category of $D^-$-stacks to the category of (un-derived) stacks
\[
t_0: D_-St(k)\to St(k),
\] induced by the embedding of the category of commutative $k$-algebras in the category of simplicial $k$-algebras (see Section 2 \cite{TV07}). In particular, if $\cX$ is a derived Artin 1-stack then $t_0(\cX)$ is an algebraic stack in the sense of \cite{Art74}.

We freely use the notation from \cite{HP17}. In particular, $\dCx(X)$, $\RR\Perf(X)$ and $\dVt(X)$ (and $\dVt^\gr(X)$) 
denote the derived moduli stacks of complexes, perfect complexes and vector bundles (and graded vector bundles) on $X$, respectively.

\section{Poisson structures on meromorphic endomorphism algebras}\label{sec:loopalg}

Let $E$ be a stable vector bundle over an elliptic curve $C$, and let $D$ be a divisor on $C$ of positive degree. 
There exists a homogeneous quadratic Poisson structure on $A(E,D):=\Hom(E,E(D))$, 
which is referred to as \emph{Mukai bracket} by Hurtubise and Markman (cf.\  \cite[Lemma 3.16]{HM02}). In this section, we will show that the projectivization of $A(E, D)$, with its natural Poisson structure induced by the Mukai bracket, is obtained as the classical shadow (see Section \ref{sec:shiftedPoi}) of the 0-shifted Poisson structure on the moduli spaces of complexes on $C$ constructed in \cite{HP17}. We show that as $D$ varies among all effective divisors, these Poisson brackets on $A(E,D)$
are compatible and hence give a Poisson structure on the algebra $A(E)$ of meromorphic endomorphisms of $E$. 
Furthermore, we show that the product map is Poisson.

\subsection{Shifted Poisson structure and moduli space of complexes}\label{sec:shiftedPoi}

We refer to \cite{PTVV}, \cite{CPTVV} and \cite{Sp16} for the general theory of shifted symplectic and Poisson structures.

Let $\cX$ be a derived Artin stack over $k$.  Then an $n$-shifted Poisson structure $h$ on $\cX$ defines a morphism
in the derived category
\[
\xymatrix{
\Pi_h : \LL_{\cX}[n]\ar[r] &\TT_\cX}.
\]
In the case when $n=0$, taking the $0$-th cohomology, we get a morphism
\[
H^0(\Pi_h): H^0(\LL_X)\to H^0(\TT_X).
\]
We call this map  the \emph{classical shadow} of a $0$-shifted Poisson structure. 

In \cite{HP17}, for a smooth projective CY $d$-fold we constructed a $(1-d)$-shifted Poisson structured
on the derived moduli space $\dCx(X)$ of (bounded) complexes of vector bundles over $X$. It arises by a general procedure
from a Lagrangian structure on the natural map 
$$(q,p):\dCx(X)\to \RR\Perf(X)\times \dVt^\gr(X),$$
where $q$ (resp., $p$) 
sends a complex of vector bundles to the corresponding object in $\D(X)$ 
(resp., the underlying graded vector bundle), and from $(2-d)$-shifted symplectic structures on $\RR\Perf(X)$ and
$\dVt^\gr(X)$ constructed in \cite{PTVV}.
  
In this paper we are only interested in the case $d=1$. In this case we have the following result. 

\begin{theorem} (\cite[Theorem 3.13, Corollary 3.20, Corollary 3.21]{HP17}) \label{HPmainthm}
Let $C$ be a smooth elliptic curve.
The derived stack $\dCx(C)$ has a 0-shifted Poisson structure. Given an object in $\D^b(C)$ (resp. a graded vector bundle on $C$), denote its corresponding stacky point (see Definition 2.4 \cite{HP17}) by $x$ (resp. by $y$). The homotopy fiber of $q$ at $x$ (resp. fiber of $p$ at $y$), inherits the 0-shifted Poisson structure from $\dCx(C)$. Moreover, the homotopy fiber of the map $(q,p)$ at $(x,y)$ has a 0-shifted symplectic structure.
\end{theorem}

As an analogue of the terminology in classical Poisson geometry, we call a non-empty homotopy fiber of the map $(q,p)$ a \emph{derived symplectic leaf} of the $0$-shifted Poisson structure on $\dCx(C)$.
In Section \ref{sec:symleaves} we will see that if restrict to the complexes of the form $E\to E(D)$, then
truncating the derived structure on each derived symplectic leaf and passing to the coarse moduli spaces
we get the usual symplectic leaves on the classical Poisson variety $\PP \Hom(E,E(D))$.

\subsection{Poisson structure on $M(E,D)$}

For a fixed elliptic curve $C$ let $\Trp$ denote the derived moduli stack of triples $(V_0\rTo{\phi} V_1)$ on $C$,
viewed as an open substack in $\dCx(C)$ consisting of complexes concentrated in degrees $0$ and $1$.
Note that at the point where $\Ext^1(V_0,V_1)=0$ the derived structure is trivial, so near such a point we have a
smooth Artin stack. For fixed $V_0$ and $V_1$ with $\Ext^1(V_0,V_1)=0$ the homotopy fiber of the morphism
$p:\dCx(C)\to \dVt^\gr(C)$ over $V_0\oplus V_1[-1]$ is a smooth Artin stack which we denote as $\Trp_{V_0,V_1}$.
By Theorem \ref{HPmainthm}, we get a $0$-shifted Poisson structure on $\Trp_{V_0,V_1}$.

Let us fix a stable vector bundle $E$ on $C$ of rank $r$, and let $D$ be a divisor of degree $k>0$. 
We are interested in the open substack
$$\cM(E,D)\sub \Trp_{E,E(D)}$$ 
consisting of $\phi\neq 0$, with its $0$-shifted Poisson structure.
Note that $\cM(E,D)$ is a trivial $\GG_m$-gerbe over its coarse moduli space, denoted by $M(E,D)$, which is isomorphic to 
the projective space $\PP \Hom(E,E(D))$. 
{\bf{Unless different bundle $E$ and divisor $D$ are considered at the same time, we will denote $\cM(E,D)$ simply
by $\cM$.}}

\begin{prop}
The 0-shifted Poisson structure on $\cM$ descends to its coarse moduli space $M=M(E,D)$.
\end{prop}
\begin{proof}
By the vanishing of hypercohomology, the derived structure on $\cM$ is trivial. Because $f:\cM\to M$ is a $\GG_m$-gerbe,
\[
f^*T_{M} \cong H^0(\TT_{\cM}), ~~f^*\Omega^1_{M} \cong H^0(\LL_{\cM}).
\]
Hence, the classical shadow of $\Pi$ is a morphism $H^0(\Pi):  f^*\Omega_{M} \to f^*T_{M}$. But $f_*\cO_{M} \cong \cO_{\cM}$, since $f$ is a $\GG_m$-gerbe, so $H^0(\Pi)$ descends to a morphism $\Omega_{M} \to T_{M}$.
\end{proof}

Let us also set
$$A(E,D):=\Hom(E,E(D)).$$
We are going to show that the Poisson structure on $M(E,D)$ lifts in a canonical way to a quadratic Poisson structure
on a vector space $A(E,D)$ (see Theorem \ref{poisson-bivector-thm} below).

First, we are going to give an explicit formula for the Poisson structure on $\cM$. 

The tangent complex  of $\cM^\pp_{r,d}$ at $(E,E(D),\phi)$, denoted by $\TT_\phi$,  is quasi-isomorphic to $\Gamma(C^{cos},\cC^\bullet[1])$, where $C^{cos}$ is a co-simplicial resolution of $C$ and $\cC^\bullet$ is the complex of vector bundles
\[\xymatrix{
\sEnd(E)\oplus\sEnd(E)\ar[r]^\partial & \sEnd(E)(D),}
\]
with differential $\partial$  defined by
\[
\partial(\alpha,\beta)=\alpha\phi -  \phi \beta
\] for any local sections $(\alpha,\beta)$  of $\sEnd(E)\oplus\sEnd(E)$. 
Using Cech model, $\Gamma(C^{cos},\cC^\bullet[1])$ is taken to be the Cech cochain complex that computes the hypercohomology of $\cC^\bullet$.

The cotangent complex, denoted by $\LL_\phi$, is quasi-isomorphic to $\Gamma(C^{cos},(\cC^\bullet)^\vee)$. 
The complex $(\cC^\bullet)^\vee[-1]$ can be written as
\[\xymatrix{
\sEnd(E^\vee)(-D)\ar[r]^{-\partial^\vee} & \sEnd(E^\vee)\oplus\sEnd(E^\vee)}
\]
where $-\partial^\vee(\psi)=(-\phi\psi,\psi\phi)$ for a local section $\psi$.

It is computed in Theorem 4.7 of \cite{HP17} that the 0-shifted Poisson structure $\Pi_h$ (we will omit the homotopy $h$ and write $\Pi$ instead) on $\dCx(C)$ is induced by the chain map
\[
\pi:=(0,\partial\circ\ft): (\cC^\bullet)^\vee[-1]\to \cC^\bullet
\]
where $\ft:=t_0-t_1 : \sEnd(E^\vee)\oplus\sEnd(E^\vee)\to \sEnd(E)\oplus\sEnd(E)$ with $t_i$ is the auto duality isomorphism
\[
t_0: \sEnd(E^\vee)\to \sEnd(E),~~ t_1: \sEnd((E(D))^\vee)\to \sEnd(E(D)).
\]
Let $\psi$ and $(a,b)$ be local sections of $\sEnd(E^\vee)(-D)$ and $\sEnd(E^\vee)\oplus\sEnd(E^\vee)$, we have
\begin{equation}\label{Poi-0}
\pi(\psi,(a,b))= (0,a\phi+\phi b ).
\end{equation}

The determinant map induces an isomorphism between $\cU_{r,d}$ and the Picard stack $Pic^d(C)$. Composing it with $p$, we get
\[
\det\circ p: \cM^\pp_{r,d}\to \cU_{r,d}\times \cU_{r,d}\to Pic^d(C)\times Pic^d(C).
\]
For simplicity of notations, we denote $\sEnd(E)$ by $\fg$ and the sheaf of trace-free sections $\sEnd_0(E)$ by $\fs$. 

The stupid truncation of $\cC^\bullet$ and $(\cC^\bullet)^\vee[-1]$ gives two long exact sequences of cohomology groups
\begin{equation}\label{les1}
\xymatrix{
0\ar[r] & \HH^0(\cC^\bullet)\ar[r]^{p_*} &H^0(\fg)^{\oplus 2}\ar[r] &H^0(\fg(D))\ar[r]&\HH^1(\cC^\bullet)\ar[r]^{p_*} & H^1(\fg)^{\oplus 2}\ar[r] &0
}
\end{equation}
\begin{equation}\label{les2}
\xymatrix{
0\ar[r] & H^0(\fg^\vee)^{\oplus 2}\ar[r]^{p^*} &\HH^0((\cC^\bullet)^\vee)\ar[r]  &H^1(\fg^\vee(-D))\ar[r]&H^1(\fg^\vee)^{\oplus 2}\ar[r]^{p^*} &  \HH^1((\cC^\bullet)^\vee)\ar[r] &0
}
\end{equation}
The tangent map of the determinant map is
\begin{align*}
(\tr,\tr): H^0(\fg)^{\oplus 2}\to H^0(\cO_C)^{\oplus 2}.
\end{align*}
Because $E$ is stable, this is an isomorphism
with inverse $(\frac{\tr}{r}\Id,\frac{\tr}{r}\Id)$. We identify $\HH^i(\fg)$ (resp. $\HH^i(\fg^\vee)$) in the exact sequences with $H^i(\cO_C)$ by this isomorphism. 

Let $(U_+,U_-)$ be an open affine covering of $C$ and let us set $U_\pm:=U_+\cap U_-$. 
A covector of $\cM$ is an element of cokernel of $p^*: H^0(\fg^\vee)^{\oplus 2}\to\HH^1((\cC^\bullet)^\vee[-1])$. 
It can be represented by a cocycle $(\psi_\pm,a_+,a_-,b_+,b_-)$, where $\psi_\pm\in \fg(-D)(U_\pm)$, $a_+,b_+\in \fg(U_+)$, $a_-,b_-\in \fg(U_-)$. The cocycle condition is 
\begin{align} \label{cocycOmega}
-\phi\psi_\pm=a_+-a_-,&&\psi_\pm\phi=b_+-b_-.
\end{align} 
This cocycle is considered up to coboundary and up to adding
elements of the form $(0, \lambda\cdot \Id,\lambda\cdot \Id,\mu\cdot \Id,\mu\cdot \Id)$ for any scalars $\lambda,\mu$.
Let $a$ be a local section of $\fg$.
Define 
\[
\pr(a):=a-\frac{\tr(a)}{r}\Id. 
\]
Clearly $\pr(a)$ is a local section of $\fs\subset \fg$. By stability of $E$, we have $H^i(\fs)=0$ for $i=0,1$. Hence, there is a 
canonical isomorphism 
\[
\Gamma(U_+,\fs)\oplus\Gamma(U_-,\fs)\cong \Gamma(U_\pm,\fs).
\]
Let us denote by $P_+:\Gamma(U_\pm,\fs)\to \Gamma(U_+,\fs)$ and $P_-:\Gamma(U_\pm,\fs)\to \Gamma(U_-,\fs)$ the corresponding projections. Using the decomposition $\fg\cong\fs\oplus \cO_C$  we can rewrite equation \eqref{cocycOmega} as
\begin{align*} 
\pr(a_+)=-P_+(\pr(\phi\psi_\pm)),&& \pr(b_+)=P_+(\pr(\psi_\pm\phi)),\\
\pr(a_-)=P_-(\pr(\phi\psi_\pm)),&& \pr(b_-)=-P_-(\pr(\psi_\pm\phi)),\\
\tr(a_+-a_-)=-\tr(\phi\psi_\pm),&& \tr(b_+-b_-)=\tr(\psi_\pm\phi).
\end{align*} 

Now we can rewrite the expression in formula \eqref{Poi-0} as
\begin{align*} 
a_+\phi+\phi b_+&=(\pr(a_+)+\frac{\tr(a_+)}{r}\Id)\phi+\phi(\pr(b_+)+\frac{\tr(b_+)}{r}\Id)\\
&=[-P_+(\pr(\phi\psi_\pm))\phi+\phi P_+(\pr(\psi_\pm\phi))]+\frac{\tr(a_++b_+)}{r}\phi.
\end{align*} 
Recall that $\Pi_\phi$ can be viewed as a linear map with the target $H^0(\fg(D))/\lan\phi\ran$, so the above formula
gives
 \begin{equation}\label{Poi-2}
\Pi_\phi(\psi_\pm,a_+,b_+)\equiv -P_+(\pr(\phi\psi_\pm))\phi+\phi P_+(\pr(\psi_\pm\phi)) \mod \lan\phi\ran.
\end{equation}

Let us consider the bivector $\wt{\Pi}$ on $\cA(E,D)$ corresponding to the linear map
$H^1(\fg(-D))\to H^0(\fg(D))$ given by
\begin{equation}\label{Poi-3}
\wt{\Pi}(\psi_\pm):= -P_+(\pr(\phi\psi_\pm))\phi+\phi P_+(\pr(\psi_\pm\phi)).
\end{equation}

Thus, \eqref{Poi-2} is an explicit formula for the Poisson bivector $\Pi$ on $M(E,D)$
while \eqref{Poi-3} defines a bivector $\wt{\Pi}$ on $A(E,D)$, which is compatible with $\Pi$
via the natural projection $A(E,D)\setminus\{0\}\to M(E,D)$.
Below we will prove that in fact $\wt{\Pi}$ defines a Poisson structure on $A(E,D)$.

\subsection{Determinant map and Poisson structure on $A(E,D)$}\label{sec:det}

Before turning to the bivector $\wt{\Pi}$ on $A(E,D)$ we make some general observations about Poisson brackets.

\begin{definition} Let $X$ be a scheme equipped with a morphism $h:\Om_X\to \cT_X$, such that the corresponding
bracket on $\cO$ given by $\{\phi,\psi\}=\lan d\phi,h(d\psi)\ran$ is skew-symmetric. A morphism $f:X\to Y$,
where $Y$ is a scheme with no extra structure, is called 
{\it Casimir} if pull-backs of local functions on $Y$ have zero brackets with functions on $X$. Equivalently,
the composition $f^*\Om_Y\to \Om_X\rTo{h} \cT_X$ should be zero.
\end{definition}

If $X$ and $Y$ are smooth varieties then the condition that $f:X\to Y$ is Casimir with respect to a bivector $h$ 
is equivalent to saying that for every $x\in X$, the composition 
$$T^*_x X\rTo{h_x} T_x X\rTo{df_x} T_{f(x)}Y$$
is zero. In the case $f$ is smooth this is equivalent to the fact that $h:\Om_X\to\cT_X$ factors as
$$\Om_X\to \Om_{X/Y}\to \cT_{X/Y}\to \cT_X$$
where $\cT_{X/Y}$ (resp., $\Om_{X/Y}$) is the relative tangent (resp., cotangent) bundle.

It is easy to see that if $f:X\to Y$ is a Casimir morphism that factors through a closed subscheme $Z\sub Y$ then
the corresponding morphism $X\to Z$ is still Casimir.
 
In the case when $X$ is a Poisson scheme, and $f:X\to Y$ is a Casimir morphism then every fiber of $f$ is a 
Poisson subscheme. 

\begin{lemma}\label{poisson-Casimir-lem}
Let $f:X\to Z$ be a smooth morphism of smooth varieties, and let $h\in \Gamma(X,\bigwedge^2 \cT_X)$ be a bivector on $X$ such that $f$ is Casimir. Then $h$ defines a Poisson structure if and only if for every $z\in Z$ the induced
bivector $h_z$ on $f^{-1}(z)$ defines a Poisson structure on $f^{-1}(z)$.
\end{lemma}

\begin{proof}
The condition that $h$ defines a Poisson structure is equivalent to the identity
\begin{equation}\label{Jacobi-identity-omega}
h(\om_1)\cdot \lan h(\om_2),\om_3\ran-\lan [h(\om_1),h(\om_2)],\om_3\ran + c.p.(1,2,3)=0
\end{equation}
for any local $1$-forms $\om_1,\om_2,\om_3$ on $X$ (where the omitted terms are obtained
by cyclic permutation of the indices $1,2,3$). The fact that $f$ is Casimir means that
the vector fields $h(\om_i)$ belong to the relative tangent bundle $\cT_{X/Z}\sub \cT_X$. Hence, for a local
function $\phi$ on $X$, the value $h(\om_i)\cdot \phi$ at $x\in X$ depends only on the restriction of $\phi$ to
the fiber of $f$ containing $x$. Similarly, the value of the vector field $[h(\om_i),h(\om_j)]$ at $x$ depends
only on the restrictions $\om_i|_{f^{-1}(f(x))}$, $\om_j|_{f^{-1}(f(x))}$. It follows that \eqref{Jacobi-identity-omega}
holds if and only if the same identity holds for local $1$-forms on every fiber $f^{-1}(z)$ (with $h$ replaced by $h_z$).
\end{proof}

We have the following general result about Poisson structures.

\begin{prop}\label{gen-casimir-prop}
Let $p:X\to \ov{X}$ and $f:X\to Z$ be smooth surjective maps of smooth varieties, and let $h\in \bigwedge^2\cT_X$,
$\ov{h}\in \bigwedge^2 \cT_{\ov{X}}$ be compatible bivectors. Assume that $f$ is Casimir with respect to $h$,
and that for every $x\in X$ the map
$$T_x X\rTo{(dp,df)} T_{p(x)}\ov{X}\oplus T_{f(x)}Z$$
is injective. Then $h$ defines a Poisson structure on $X$ if and only if $\ov{h}$ defines a Poisson structure on $\ov{X}$.
\end{prop}

\begin{proof} The fact that $h$ and $\ov{h}$ are compatible means that for evey $x\in X$, the map 
$\ov{h}:T^*_{p(x)}\ov{X}\to T_{p(x)}\ov{X}$
factors as the compoisition
$$T^*_{p(x)}\ov{X}\rTo{p^*} T^*_x X\rTo{h} T_xX\rTo{dp} T_{p(x)}\ov{X}.$$
It follows that for any local $1$-form $\om$ on $\ov{X}$ the vector field $h(p^*\om)$ on $X$ belongs to the subsheaf
\[
\cT_p:=(dp)^{-1}(p^{-1}\cT_{\ov{X}})\sub \cT_X \]
and has the property $dp(h(p^*\om))=\ov{h}(\om)$.
Furthermore, the natural projection $dp:\cT_p\to p^{-1}\cT_{\ov{X}}$ is compatible with the Lie brackets, so we have
$$dp([h(p^*\om_1),h(p^*\om_2)])=[\ov{h}(\om_1),\ov{h}(\om_2)].$$
Using this it is easy to see that the identity \eqref{Jacobi-identity-omega} for the bivector $h$ and
the $1$-forms $p^*\om_1$, $p^*\om_2$, $p^*\om_3$ on $X$ is equivalent to the same identity for the bivector $\ov{h}$
and the $1$-forms $\om_1,\om_2,\om_3$ on $\ov{X}$.
This implies the ``only if" part (for this part we do not need the map $f$).

Now let us prove the ``if" part. 
By Lemma \ref{poisson-Casimir-lem}, it is enough to check that for every $z\in Z$ the induced bivector $h_z$ 
on $f^{-1}(z)\sub X$ defines a Poisson structure on $f^{-1}(z)$. But our assumption implies that the restricted map
$$p|_{f^{-1}(z)}: f^{-1}(z)\to \ov{X}$$ 
induces embedding on tangent spaces. Since this map is compatible with the
bivectors $h_z$ and $\ov{h}$ and the latter defines a Poisson structure on $\ov{X}$, it follows that $h_z$ defines a 
Poisson structure on $f^{-1}(z)$. Indeed, the closure of the image of $p|_{f^{-1}(z)}$ is a Poisson subvariety of $\ov{X}$,
so we can restrict to its open subset over which $p|_{f^{-1}(z)}$ is \'etale, and identify the induced bivector
with the \'etale pull-back of a Poisson structure, which is always Poisson.
\end{proof}

It turns out that the determinant map $\det: A(E,D)\to H^0(E,\cO(rD))$ is Casimir,
where we equip $A(E,D)$ with the bivector $\wt{\Pi}$ given by \eqref{Poi-3}.

\begin{prop}\label{detcasimir}
The determinant map $\det:A(E,D)\to H^0(E,\cO(rD))$ is Casimir with respect to $\wt{\Pi}$. As a consequence, its projectivization 
$M^{\reg}(E,D)\to \PP H^0(rD)$ is Casimir with respect to $\Pi$.
\end{prop}
\begin{proof}
 Let $A$ and $B$ be two $r \times r$ matrices. Recall that the derivative of the determinant map at $A$ along $B$ is equal to $\tr(B \cdot \adj(A))$ where $\adj(A)$ refers to the adjugate matrix of $A$. Now the assertion follows by a direct calculation:
\begin{align*}
{\det}_*\wt{\Pi}_\phi(\psi) &= \tr\Big([\phi P_+(pr(\psi\phi))- P_+(pr(\phi \psi))\phi]\cdot \adj(\phi)\Big)\\
&= \tr\Big(\adj(\phi)\phi\cdot P_+(pr(\psi\phi))\Big) - \tr\Big(P_+(pr(\phi\psi))\cdot\phi\cdot\adj(\phi)\Big)\\
&= \det(\phi)\cdot \tr\big(\phi P_+(pr(\psi\phi))\big) - \det \phi\cdot \tr\big(P_+(pr(\phi\psi))\big)= 0.
\end{align*}
\end{proof}

\begin{theorem}\label{poisson-bivector-thm}
The bivector $\wt{\Pi}$ given by \eqref{Poi-3} defines a quadratic Poisson structure on $\cA(E,D)$. 
\end{theorem}

\begin{proof}
Let $Z\sub H^0(E,\cO(rD))$ be the closure of the image of the map $\det$, and let $Z'\sub Z\setminus \{0\}$ 
be a dense open subset
such that the map $\det$ is smooth and surjective over $Z'$. Thus, if $X\sub A(E,D)$ is the preimage of $Z'$
then we can apply Proposition \ref{gen-casimir-prop} to the projection $p:X\to \PP X$, the morphism
$\det:X\to Z'$ and the bivectors $\wt{\Pi}$ and $\Pi$. By Proposition \ref{detcasimir}, the morphism $\det$ is Casimir. Since
we know that $\Pi$ gives a Poisson structure on $\PP X$, we deduce that $\wt{\Pi}$ gives a Poisson structure on $X$.
Thus, the Schouten-Nijenhuis bracket $[\wt{\Pi},\wt{\Pi}]$ vanishes on a dense open subset $X\sub A(E,D)$.
Hence, it vanishes identically.
\end{proof}

\begin{remark}
Theorem \ref{poisson-bivector-thm} can also be deduced from the results of Hurtubise and Markman \cite[Sec.\ 3]{HM02}.
They also proved that the Poisson structure associated to formula \eqref{Poi-3} is induced
by the Sklyanin bracket on the loop group defined by Belavin's elliptic $r$-matrix (\cite[Theorem 3.24]{HM02}). 
\end{remark}


\begin{remark}
For any two stable vector bundles $E$ and $E^\p$ of rank $r$, there exists a line bundle $L$ and an isomorphism 
$E'\simeq E\otimes L$. This leads to a canonical identification $\sEnd(E')\simeq \sEnd(E)$ and hence,
$A(E,D)\cong A(E^\p,D)$.  Using formula \eqref{Poi-3} it is easy to see that
this isomorphism is compatible with Poisson brackets.
Therefore, up to an isomorphism, the Poisson varieties $A(E,D)$ and $M(E,D)$ depend only on the rank $r$
and the isomorphism class of the line bundle $\cO(D)$. 
\end{remark}

\begin{remark}\label{Rel-Skl-RS-Chered}
If we fix realization of the elliptic curve $C$ and a stable vector bundle $E$, one can compute the Poisson bracket associated to the bivector \eqref{Poi-3} explicitly in terms of theta functions. Fix a lattice $\Gamma=\ZZ+\ZZ\tau$ such that $C\cong \CC/\frac{1}{r}\Gamma$ and denote $\wt{C}$ for $\CC/\Gamma$. There exists a line bundle $\wt{L}$ on $\wt{C}$ and an irreducible representation $W_{r,d}$ of Heisenberg group of order $r^3$ such that $E$ is the quotient of $\wt{L}\ot W_{r,d}$ by the torsion subgroup $\frac{1}{r}\Gamma/\Gamma$. For simplicity, we consider $D=\{e\}$, being the neutral element of $C$. 
Then $A(E, nD)$ admits a canonical basis (see Section 11.1 \cite{RSTS}). Under this basis, the kernel of the projection operator $P_+$ in formula \eqref{Poi-3} turns out to be the Belavin elliptic $r$-matrix. Moreover, using Cech covering we can compute the associated bracket of \eqref{Poi-3} explicitly (similar to Section 5 \cite{HP17}). In this way, we recover the calculation of bracket of Sklyanin in the case of rank 2 and simple pole (\cite{Skl82}), Reyman and Semenov-Tian-Shansky in the case of rank 2 and arbitrary pole (\cite[Section 12]{RSTS}) and Cherednik in the general case (\cite{Ch1}). The explicit form of these brackets is used to derive their quantization, given by the  \emph{generalized Sklyanin algebras} \cite{Skl82} \cite{Ch2}. It would be interesting to find out whether our coordinate-free approach leads to a coordinate-free construction of generalized Sklyanin algebras. We leave this question for future research.
\end{remark}

\subsection{Multiplication map}

Now let us fix a pair of divisors of positive degree, $D$ and $D'$.
We have the natural multiplication map 
$$\mu: A(E,D)\times A(E,D')\to A(E,D+D'):\phi_1\otimes\phi_2\mapsto \phi_1\circ\phi_2.$$ 

\begin{prop}\label{prod-Poi}
The map $\mu$ is Poisson with respect to the Poisson structures given by \eqref{Poi-3}.
\end{prop}
\begin{proof}
Fix $a\in \cA(E,D)$ and $b\in \cA(E,D')$ and set $\phi:=\mu(a,b)\in \cA(E,D+D')$. The tangent map 
\[
\mu_*: \fg(D)\oplus \fg(D')\to \fg(D+D')
\] is $\mu_*(\phi_i,\phi_j)=\phi_ib+a\phi_j$. The cotangent map 
\[
\mu^*: \fg^\vee(-D-D')\to \fg^\vee(-D)\oplus \fg^\vee(-D')
\] is $\mu^*(\psi)=(b\psi, \psi a)$. We need to compute the composition 
\[
\xymatrix{
\fg^\vee(-D-D')\ar[r]^{\mu^*} &\fg^\vee(-D)\oplus \fg^\vee(-D')\ar[r]^{\hspace{1cm}\wt{\Pi}} &\fg(D)\oplus \fg(D')\ar[r]^{\mu_*} & 
\fg(D+D').
}
\]
Apply formula \eqref{Poi-3},
\begin{align*}
\mu_*(\wt{\Pi}(\mu^*(\psi)))&=\mu_*(\wt{\Pi}(b\psi,\psi a))\\
&=\mu_*(-P_+(\pr(ab\psi))a+aP_+(\pr(b\psi a)), -P_+(\pr(b\psi a))b+bP_+(\pr(\psi ab)))\\
&=-P_+(\pr(\phi\psi))\phi+aP_+(\pr(b\psi a))b-aP_+(\pr(b\psi a))b+\phi P_+(\pr(\psi\phi))\\
&=\wt{\Pi}(\psi)
\end{align*}
\end{proof}

Note that since the Poisson structure on each $\cA(E,D)$ is $\GG_m$-invariant, 
the product map descends to a Poisson morphism
$$M(E,D)\times M(E,D')\to M(E,D+D').$$

In the case when the divisors $D_1$ and $D_2$ are effective and $D_1\le D_2$, we can view $\sEnd(E)(D_1)$ naturally as a 
subsheaf of $\sEnd(E)(D_2)$. This gives
rise to embeddings
\[
A(E,D_1)\hookrightarrow A(E,D_2), \ \ M(E,D_1)\hookrightarrow M(E,D_2).
\] 

\begin{corollary}
For $0\le D_1\le D_2$, the embedding $\cA(E,D_1)\to \cA(E,D_2)$ (resp., $M(E,D_1)\to M(E,D_2$)
is a Poisson map.
\end{corollary}
\begin{proof} This follows from Proposition \ref{prod-Poi} together with the fact that the Poisson structure on
$\cA(E,D)$ vanishes on the identity map $\Id_E$. The latter fact can be seen immediately from formula \eqref{Poi-3}.
\end{proof}

Thus, the ind-schemes
\[
A(E):= \varinjlim_{D\ge 0} A(E,D) \text{  and  } M(E):=\varinjlim_{D\ge 0} M(E,D)
\]
can be equipped with natural Poisson structures in the sense of \cite[Sec.\ 3.2]{Williams13}.

\section{Symplectic leaves}\label{sec:symleaves}
Unlike in smooth Poisson geometry where the symplectic leaves are defined as equivalent classes of points that are connected by Hamiltonian flow, in algebraic geometry we define a symplectic leaf of a Poisson scheme to be a smooth connected Poisson subscheme that is symplectic. Therefore, the existence of symplectic leaves in a Poisson scheme is not guaranteed  from the definition. The results in \cite{HP17} provide a large class of examples of Poisson schemes that admit symplectic foliations. 

In this section we will classify the symplectic leaves of $M=M(E,D)$ and study their families over the strata of a
certain parameter space. Our results can be summarized as follows.
From Theorem \ref{HPmainthm}, we get a foliation of $\cM=\cM(E,D)$ by derived symplectic leaves. In Section \ref{sec:smoneleaf}, we prove that the coarse moduli space of a derived symplectic leaf is a smooth scheme (see Theorem \ref{smoothfiber}). In Section \ref{sec:classifyleaf}, we provide a rough classification result of leaves by certain combinatorial data (partitions) (see Theorem \ref{lem_classleaf1}, Proposition  \ref{connected-prop}). 
In Section \ref{sec:deformleaf}, we study the family of leaves
of the fixed combinatorial type over a certain stratum of the symmetric product of $C$  (see Theorem \ref{fibrationleaf}).    In Section \ref{sec:multleaf}, we prove a decomposition formula for the product of two leaves under the multiplication map (see Theorem \ref{multleaf}).

Throughout this section we fix an elliptic curve $C$ and a stable bundle $E$ of rank $r$.
Everywhere except Sec.\ \ref{sec:multleaf} we also fix a divisor $D$ of degree $k>0$ on $C$.

\subsection{From a derived symplectic leaf to its coarse moduli}\label{sec:smoneleaf}
Denote $\coh(C)$ for the category of coherent sheaves on $C$. We will use the same notation for the stack of objects in this category.
Let $\cH=(\cH^0,\cH^1)\in \coh(C)\times \coh(C)$ be a pair of coherent sheaves on $C$. Because $C$ has dimension one, any object in $\D^b(\coh(C))$ is quasi-isomorphic to the direct sum of the shifts of its cohomology sheaves. Let 
$F_\cH$ denote the homotopy fiber of $q:\cM\to \RR\Perf(C)$ over the stacky point $\cH^0\oplus\cH^1[-1]$. 
By Theorem \ref{HPmainthm}, $F_\cH$ carries a 0-shifted symplectic structure, i.e. $F_\cH$ is a derived symplectic leaf of $\cM$. Note that even though $\cM$ carries trivial derived structure, the symplectic leaf $F_\cH$ has a nontrivial derived structure
(otherwise, it could not be symplectic, since it also has a nontrivial stacky structure). 
We denote the coarse moduli space of the truncation $t_0(F_\cH)$ by $F_\cH^c$. In general, the coarse moduli space of a (classical) Artin 1-stack, if it exists, is an algebraic space. 

We will constantly quote the following lemmas about stacks.

\begin{lemma}\cite[section 4]{Vez10}\label{obs}
Let $\cX$ be a derived Artin 1-stack and denote $\LL_\cX$ for its cotangent complex. Denote
\[
j: t_0(\cX)\to \cX
\] for the natural morphism from the truncation $t_0(\cX)$ to $\cX$. Then $j^*\LL_\cX$ defines an obstruction theory on $t_0(\cX)$.
\end{lemma}
For the definition of obstruction theory on 1-stack (more generally, groupoid-valued functor), we refer to Definition 2.6  of \cite{Art74}. Given a groupoid valued functor $\cF$ with an obstruction theory and  an object $a\in \cF(A)$, for any infinitesimal extension $A^\p\to A$ there exists a class $o_a(A^\p)$ in certain obstruction sheaf so that $o_a(A^\p)$ vanishes if and only if  the  fibered category $\ol{\cF}_a(A^\p)$ (see Section 1 of \cite{Art74} for the definition) is non-empty.  We say $\cF$ is \emph{un-obstructed} if for every infinitesimal extension $A^\p\to A$, the induced map $\ol{\cF}(A^\p)\to \ol{\cF}(A)$ is surjective. Here $\ol{\cF}(A)$ is the set of isomorphism classes of the groupoid $\cF(A)$. And this holds if and only if the obstruction class $o_a(A^\p)$ vanishes for any $a\in \cF(A)$. Note that an obstruction theory of a groupoid-valued functor only depends on its underlying set-valued functor. In Lemma \ref{obs}, the obstruction sheaf on $t_0(\cX)$ is $H^2(j^*\LL_\cX)$.

\begin{lemma}\label{smoothness}
Let $\cG\to \cF$ be a surjective morphism of algebraic 1-stacks locally of finite type, equipped with obstruction theories.
\footnote{For an algebraic stack, there always exists an obstruction theory satisfying the conditions in Artin's representability theorem. We refer to Section 5 \cite{Art74} for details.}
Assume that $\cG$ is un-obstructed, and that the induced map on $H^0$ of the tangent complexes is surjective. Then $\cF$ is un-obstructed. In particular, $\cF$ is smooth.
\end{lemma}

The proof will be based on the following standard result in deformation theory (for completeness, we give a proof).
\paragraph{Sublemma}(c.f. Ex. 15.8 of \cite{HarDef})
Let $G\to F$ be a morphism of set-valued functors of local Artin rings. Assume that $G$ and $F$ admit mini-versal families,  $G$ is un-obstructed, and that the tangent map is surjective. Then $G\to F$ is strongly surjective and $F$ is un-obstructed. 

\begin{proof}
Denote $\cC$ for the category of local Artin $k$-algebras with residue field $k$. Given $A\in \cC$, we first show that $G(A)\to F(A)$ is surjective. Denote $D$ for the $k$-algebra of dual number.
The sublemma can be proved by induction on the dimension of $A$. By versality, surjectivity holds when $A=k$. Suppose the surjectivity is proved for $A$. Let $A^\p\to A$ be a small extension.  Consider the commutative diagram
\begin{equation}\label{strongsurj}
\xymatrix{
G(A)\ar[r] & F(A)\\
G(A^\p)\ar[u]^q\ar[r]^{f^\p} & F(A^\p)\ar[u]^p 
}
\end{equation}
By the un-obstructedness of $G$, $q$ is surjective. Fix  $\eta^\p\in F(A^\p)$ and define $\eta:=p(\eta^\p)$. By the induction assumption and the surjectivity of $q$, there exists $\xi^\pp\in G(A^\p)$ with image $\eta^\pp\in F(A^\p)$ so that $p(\eta^\pp)=\eta$. By mini-versality (Proposition 16.1 \cite{HarDef}), the action of $F(D)$ on $p^{-1}(\eta)$ is transitive. Let $v_F\in F(D)$ be an element so that 
$v_F(\eta^\pp)=\eta^\p$.
Because $G(D)\to F(D)$ is surjective, we may choose a pre-image $v_G\in G(D)$ so that 
$f^\p(v_G(\xi^\pp))=\eta^\p$.
This proves the surjectivity of $G(A^\p)\to F(A^\p)$.

Given a small extension $B\to A$, we need to show the morphism $h$ in the diagram
\[
\xymatrix{
G(B)\ar@/^1.0pc/[rrd]^{f^\p}\ar@/_2.0pc/[rdd]^{p^\p}\ar@{-->}[rd]^{h} & & \\
& G(A)\times_{F(A)}F(B)\ar[r]\ar[d] & F(B)\ar[d]^p\\
&G(A)\ar[r]^f & F(A)
}
\] is surjective. Fix $\xi_A\in G(A),\eta_B\in F(B), \eta_A\in F(A)$ so that $f(\xi_A)=p(\eta_B)=\eta_A$. By surjectivity of $p^\p$ and $f^\p$, there exists $\xi_B,\xi^\p_B\in G(B)$ so that $f^\p(\xi_B)=\eta_B, p^\p(\xi^\p_B)=\xi_A$. Because both $f^\p(\xi_B)$ and $f^\p(\xi^\p_B)$ belong to $p^{-1}(\eta_A)$, there exists $v_F\in F(D)$ such that $v_F(f^\p(\xi_B))=f^\p(\xi^\p_B)$ by mini-versality of $F$. Because $G(D)\to F(D)$ is surjective, there exists a lift $v_G\in G(D)$ of $v_F$ so that $v_G(\xi_B)=\xi^\p_B$. The surjectivity of $h$ is proved by induction.
The un-obstructedness of $F$ follows from diagram \ref{strongsurj}.
\end{proof}

\begin{proof}[Proof of Lemma \ref{smoothness}]
It is enough to check the un-obstructedness of $\cF$ for infinitesimal extensions of residue fields of points (see Proposition 4.2 of \cite{Art74}). Thus, we can fix a point of $\cF$ and a point of $\cG$ over it, and restrict our functors
to the category $\cC$ of local Artin rings. 
Let us denote by $G$ and $F$ the underlying set-valued functors. By Artin's representability theorem (Theorem 5.3 \cite{Art74}),  the functors $F$ and $G$ satisfy  the Schlessinger's criterion (c.f. condition (S1), (S2) of Section 2 \cite{Art74}), therefore admit mini-versal families. Hence, the sublemma implies the un-obstructedness of $\cF$.
\end{proof}

The next lemma is an easy exercise. We omit the proof.
\begin{lemma}\label{gerbe}
Let $\cX$ be a (un-derived) algebraic 1-stack. Suppose $\cX$ is a $G$-gerbe over an algebraic space $X$ for some affine algebraic group $G$.
Then $X$ is isomorphic to the coarse moduli space of $\cX$. Moreover, $\cX$ is smooth if and only if $X$ is smooth.
\end{lemma}

The following is the main result of this section.

\begin{theorem}\label{smoothfiber}
Let $(\cH^0,\cH^1)$ be a pair of coherent sheaves on $C$ such that $\det(\cH^1)\otimes \det(\cH^0)^{-1}\cong \cO(rD)$. For $\cH:=\cH^0\oplus\cH^1[-1]$ in $\RR\Perf(C)$ such that $F_\cH$ is non-empty, the coarse moduli scheme $F_\cH^c$ is a smooth symplectic scheme. Denote $h^i_{jk}$ for the dimension of $\Ext^i(\cH^j,\cH^k)$ for $i,j,k=0,1$. Then the dimension of $F^c_\cH$ is equal to $r^2 k-h^0_{00}-h^0_{11}-h^0_{01}$.
\end{theorem}
\begin{proof}
Denote the stacky point in $\RR\Perf(C)$ corresponding to $\cH^0\oplus\cH^1[-1]$ by $x_\cH$. Recall that $F_\cH$ is defined to be the homotopy fiber product 
\begin{equation}\label{fiberdig}
\xymatrix{
F_\cH\ar[d]\ar[r]^j&\cM\ar[d]^q\\
x_\cH\ar[r]^i &  \RR\Perf(C)
}
\end{equation}
Let $P$ be a point in $F_\cH$ that corresponds to a triple $\phi: E\to E(D)$,
so that $\cH^0\oplus\cH^1[-1]$ is quasi-isomorphic to the complex $[E\rTo{\phi} E(D)]$. 
The tangent complex to $\RR\Perf(C)$ at this complex, has non-vanishing cohomology only in degrees 
$-1$ and $0$, because $\Ext^1(E,E(D))=0$.
On the other hand, $i$ induces an isomorphism on $H^{-1}$ of the tangent complexes.
Thus, from the long exact sequence of tangent complexes associated to the homotopy fiber product,  we deduce an
isomorphism
\[
H^{-1}(\TT_{F_\cH}|_P)\simeq H^{-1}((j^*\TT_{\cM})|_P)\simeq\CC
\] (the latter identification is due to the fact that $\cM$ is a $\GG_m$-gerbe), and a long exact sequence
\begin{equation}\label{les_fh}
\xymatrix{
0\ar[r] & H^0(\TT_{F_\cH}|_P)\ar[r] & H^0((j^*\TT_{\cM})|_P)\ar[r]^q & H^0((q\circ j)^*\TT_{\RR\Perf(C)}|_P)\\
\ar[r] & H^1(\TT_{F_\cH}|_P)\ar[r] &0
}
\end{equation}
The nondegeneracy of the shifted symplectic form on $F_\cH$ 
implies that $H^{1}(\TT_{F_\cH}|_P)\simeq\CC$, thus, the image of $q$ has codimension $1$.

Let us consider the determinant map $\det:\RR\Perf(C)\to Pic(C)$ to the Picard stack of $C$. Then the composition
$\det\circ q$ factors through the stacky point of $Pic(C)$ corresponding to $\cO(rD)$.
Since $\det$ induces a surjection on tangent spaces, it follows that the image of $q$ can be identified
with the subspace consisting of first order deformations that preserve the determinant (up to isomorphism).

Let $\RR\Perf^{rD}(C)$ denote the derived moduli space of perfect complexes on $C$ with the determinant isomorphic to $\cO(rD)$. As we observed above, $q$ factors through $\RR\Perf^{rD}(C)$. Let us denote by $G_\cH$ the homotopy fiber product 
\begin{equation}\label{fiberdig1}
\xymatrix{
G_\cH\ar[r]\ar[d]& \cM\ar[d]^q\\
x_\cH\ar[r] &  \RR\Perf^{rD}(C)
}
\end{equation}
From the previous observations we get that the tangent morphism $H^0(\TT_{\cM})\to H^0(q^*\TT_{\RR\Perf^{rD}(C)})$ is surjective. Therefore, $H^1(\TT_{G_\cH})$ vanishes.
The cotangent complex $\LL_{G_{\cH}}$ has non-zero cohomology groups concentrating in degree $0,1$. By Lemma \ref{obs}, $t_0(G_\cH)$ is un-obstructed. 

There exists a morphism $f:G_\cH\to F_\cH$ because diagram \ref{fiberdig} is Cartesian. Note that $f$ is a bijection on the sets of points. And $f$ induces isomorphisms
\[
H^{-1}(\TT_{G_\cH})\cong H^{-1}(f^*\TT_{F_\cH}),~~~~H^{0}(\TT_{G_\cH})\cong H^{0}(f^*\TT_{F_\cH}).
\] 
The morphism $f$ induces a morphism of classical stacks $t_0(G_\cH)\to t_0(F_\cH)$. Moreover, the induced morphism on $H^0$ of the tangent complexes from $t_0(G_\cH)$ to $t_0(F_\cH)$ is an isomorphism. By Lemma \ref{smoothness}, $t_0(F_\cH)$ is a smooth stack. 

We claim that $t_0(F_\cH)$ admits a coarse moduli space $F_\cH^c$. This can be proved as follows. Denote $\cB$ for the open substack of $\RR\Perf(C)$ consisting of objects $E$ of Tor amplitude $[-1,0]$ and satisfy the condition $\Ext^{<0}(E,E)=0$. Clearly, $q$ factors through $\cB$. So we may identify $F_\cH$ with the homotopy fiber product $x_\cH\times_{\cH,\cB, q}\cM$. Lieblich proved that $t_0(\cB)$ is an Artin stack (\cite{Lie05}). 
Because $t_0$ functor commutes with (homotopy) fiber product (see Section 2.2.4 \cite{HAGII}), $t_0(F_\cH)$ is equivalent to the fiber product of Artin stacks 
$$x_\cH\times_{t_0(\cB)} \cM$$
where $x_\cH$ is equivalent with $B\Aut(\cH)$, the classifying stack of the automorphism group of $\cH=\cH^{-1}[1]\oplus\cH^0$ as an object in $\D(\coh(C))$. This is a group because of the condition $\Hom(E(D),E)=0$.

It suffices to show that $t_0(F_\cH)$ is a gerbe since any gerbe admits a coarse moduli space. Recall that an algebraic stack $\cX$ is a gerbe if and only if the natural map from the inertia stack $I_\cX\to\cX$ is flat and locally of finite presentation (Proposition 91.27.9 \cite{St}). Because $\cM$ is a $\GG_m$-gerbe, the morphism $I_\cM\to\cM$ is flat and locally of finite presentation. By Lemma 91.5.6 \cite{St}, the diagram 
\[
\xymatrix{I_{t_0(F_\cH)}\ar[r]\ar[d] & I_\cM\ar[d]\\
t_0(F_\cH)\ar[r] &\cM}
\] is Cartesian if $t_0(F_\cH)\to \cM$ is a monomorphism. Recall that a representable morphism of stacks $\cX\to \cY$ is a monomorphism if $\cX\simeq\cX\times_\cY\cX$. Since monomorphism is stable under base change, it suffices to check $B\Aut(\cH)\to t_0(\cB)$ is a monomorphism. From the definition of stacky point, this is clearly a monomorphism of fibred categories. So we just need to show it is represented by algebraic spaces, which is a consequence of the fact that $t_0(\cB)$ is an algebraic stack (see \cite{Lie05}). The existence of coarse moduli space for $t_0(F_\cH)$ is then proved.

We claim that $F_\cH^c$ is a scheme. Denote the coarse moduli of $\cM$ by $M$.
Since $M$ is a scheme, we need to show $F_\cH^c\to M$ is representable. Because $t_0(F_\cH)\to\cM$ is a monomorphism of $\GG_m$-gerbes, the induced morphism $F_\cH^c\to M$ on coarse moduli spaces is a monomorphism of algebraic spaces. By \cite[0463, Lemma 27.10]{St}, a monomorphism of algebraic spaces locally of finite type is quasi-finite. Then by Stein factorization theorem of algebraic spaces (Theorem 7.2.10 \cite{Ols}), $F_\cH^c\to M$ is quasi-affine. In particular, it is representable.

Because $t_0(F_\cH)$ is a $\GG_m$-gerbe over its coarse moduli space, $F_\cH^c$ is a smooth scheme by Lemma \ref{gerbe}. 
The 0-shifted symplectic structure on $F_\cH$ descends to $F_\cH^c$ via the isomorphism $H^0(\TT_{F_\cH})\cong TF^c_\cH$.

Finally, we compute the dimension of $F_\cH^c$. The dimension of $\cM$ is equal to $r^2 k-1$. By  Serre duality, $h^0_{jk}=h^1_{kj}$. The long exact sequence \ref{les_fh} implies that the dimension of $F^c_\cH$ is equal to 
\[
r^2 k-h^0_{00}-h^0_{11}-h^0_{01}.
\]
\end{proof}

\subsection{Non-emptiness and connectedness of $F^c_\cH$}\label{sec:classifyleaf}

Since the spaces $F^c_\cH$ are defined as fiber products they can turn out to be empty.
In this section, focusing on the case when $\cH$ is of the form $(0,T)$, we determine for which torsion sheaves $T$
the space $F^c_T:=F^c_{(0,T)}$ is nonempty. We also show that these spaces $F^c_T$ are connected.

We will use the following standard notation for partitions. If $\lambda=(\lambda_1\geq \lambda_2\ldots)$ is a partition then $l(\lambda)$ is its length and $|\lambda |$ is the number of boxes in the Young diagram associated with $\lambda$. 


For every torsion sheaf $T$ of length $N$ supported on a set of distinct points $\{x_1,\ldots,x_n\}$, we associate to it a collection of partitions $\Lambda(T)=\{\lambda(x_1),\lambda(x_2),\ldots,\lambda(x_n)\}$, such that
\begin{align}\label{decompT}
T=\bigoplus_{i=1}^n \left(\bigoplus_{j=1}^{l(\lambda(x_i))} \cO_{\lambda(x_i)_j x_i}\right)
\end{align}
where
$\sum_{i=1}^n |\lambda(x_i)|=N$. Equivalently, we can rewrite this using multiplicities as
\begin{align}\label{dualdecompT}
T=\bigoplus_{i=1}^n \left(\bigoplus_{j\geq 1} \cO_{j x_i}^{\oplus m_j(\lambda(x_i))}\right).
\end{align}

Define a positive integer $l_{\max}(T)$ associated to $T$ by
\[
l_{\max}(T):=\max\{l(\lambda(x_i))|i=1,\ldots,n\}.
\]
When $T$ is supported at a single point, $l_{\max}(T)$ is simply the length of the corresponding partition.

The following lemma is well known to experts. 
\begin{prop}(\cite[Theorem 1.2]{Oda})\label{lem_pushfor}
Let $C$ be an elliptic curve and $\pi : C^\prime\to C$
be an \'etale covering of degree $r$.  If $E$ is a stable vector bundle on $C$ of rank $r$ and degree $d$, then there exists a line bundle $L^\p\in \Pic^d(C^\prime)$ such that $E\cong \pi_*(L^\p)$. Conversely, if $gcd(r, d)=1$, then for any $L^\p \in \Pic^d(C^\prime)$ the vector bundle $E:=\pi_*(L^\p)$ is stable of rank $r$ and degree $d$.
\end{prop}

\begin{theorem}\label{lem_classleaf1}
Let $T$ be a torsion sheaf on $C$. Then there exists a monomorphism
\[
\phi: E\to E(D)
\] such that $T\cong \Cok(\phi)$ if and only if $l_{\max}(T)\leq r$ and $\det(T)\cong \cO(rD)$. 
\end{theorem}
\begin{proof}
The only if part is obvious. We now prove the if part.

Fix an \'etale map $\pi: C^\p\to C$ of degree $r$ and a line bundle $L^\p$ on $C^\p$ so that $\pi_*L^\p\cong E(D)$ (by Proposition \ref{lem_pushfor}). Let $T$ be a torsion sheaf supported on a collection of points $\{x_1,\ldots,x_n\}\subset C$. Because $l_{\max}(T)\leq r$, there exists a collection of torsion sheaves $T^\p_1,\ldots,T^\p_n$ supported on the fibers $\pi^{-1}(x_1),\ldots,\pi^{-1}(x_n)$ respectively so that $T=\pi_*T^\p:=\pi_*(\bigoplus_{i=1}^n T^\p_i)$. For example, if $T$ is supported on a single point $x$ then we denote the unique partition by $\lambda$. Set $m=l(\lambda)$. Since $m\leq r$, we may choose points $x^\p_1,\ldots, x^\p_m$ in the fiber $\pi^{-1}(x)$ and torsion sheaf to be $T^\p:=\bigoplus_{i=1}^m\cO_{\lambda_ix_i^\p}$ for $i=1,\ldots,m$.

There is a canonical surjective morphism $j^\p: L^\p\to T^\p$, whose kernel is a line bundle $K^\p$ on $C^\p$. Because $\pi$ is \'etale, we get an exact sequence of sheaves on $C$
\[
\xymatrix{
0\ar[r] & \pi_*K^\p\ar[r] & E(D)\ar[r] &T\ar[r] &0 
}
\]
Since $\det(T)\cong \cO(rD)$, $\deg(\pi_*K^\p)=d$.
By Proposition \ref{lem_pushfor}, $\pi_*K^\p$ is stable. It is isomorphic to $E$ since they have isomorphic determinants.
\end{proof}
The above lemma classifies all possible cokernels when $\phi$ is assumed to be injective. 
When $\phi$ is not injective, it will not be enough to use only torsion sheaves to index the symplectic leaves. This brings some additional complexity. In next section, we give a complete classification of symplectic leaves (including the non-injective case) for rank 2 stable triples. However, we will mainly focus on the injective case.

The open subscheme in $M=M(E,D)$ (resp. substack in $\cM$) consisting of injective $\phi$ is denoted by $M^{\reg}$ (resp. $\cM^{\reg}$). 

Next, we will prove connectedness of the fibers of $q$ in $M^\reg$. Thus, $F_T^c$ are indeed symplectic leaves in $M^\reg$.

\begin{prop}\label{connected-prop}
Fix a coherent sheaf $T$. Denote the homotopy fiber of the map $q$ at $T[-1]$ by $F_T$. Its coarse moduli space $F_T^c$ is connected.
\end{prop}
\begin{proof}
We may assume that $\det(T)\simeq \cO(rD)$, since otherwise $F^c_T$ is empty.

Given an injective morphism $\phi: E\to E(D)$ with $\Cok(\phi)\cong T$, we get a surjection $E(D)\to T$. 
Conversely, starting with a surjection $\psi: E(D)\to T$, such that $\ker(\psi)$ is a stable bundle then we have
$\ker(\psi)\cong E$ (since $\ker(\psi)$ has the same determinant as $E$). Let $U\sub \PP\Hom(E(D),T)$ be
the open subset corresponding to surjective $\psi$ with stable $\ker(\psi)$. Then we have a surjective morphism
$U\to F^c_T$ associating with $\psi$ the morphism $\phi:E\simeq \ker(\psi)\to E(D)$ (well defined up to a scalar).
This implies that $F^c_T$ is irreducible.

\end{proof}
\begin{remark}
The construction of Proposition \ref{connected-prop} identifies $F^c_T$ with an open locus of the quotient stack
\begin{equation}\label{pcomp}
\overline{F}_T:=\Hom^s(E(D),T)/\Aut(T),
\end{equation}
where $\Hom^s$ stands for the space of surjective maps (points of $F^c_T$ correspond to maps with stable kernel). 
In fact, it is easy to see that the right action of $\Aut(T)$ on $\Hom^s(E(D),T)$ is free, so the coarse moduli space of $\ol{F}_T$ exists as an algebraic space. Note that the group $\Aut(T)$ can be computed using the decomposition \eqref{dualdecompT}.
For example, let us consider the case of a generic symplectic leaf when $T$ is the direct sum of skyscraper sheaves of $rk$ distinct points. In this case, $\ol{F}_T$ is the $rk$-fold product of $\PP^{r-1}$. 
In Section \ref{sec:rank-2} we will show that in the simplest case $r=2$ and $k=1$ 
the symplectic leaf $F_T^c$ is the complement to an elliptic curve in $\PP^1\times\PP^1$ of bidegree $(2,2)$.
\end{remark}

\subsection{Symplectic leaves in $A^\reg(E,D)$} \label{sec:det-sym}

We denote by $A^\reg(E,D)\sub A(E,D)$ the open subset of injective $\phi$.
The following corollary is a more precise version of some assertions stated in \cite[Sec.\ 3]{HM02}.

\begin{corollary} \label{leafofA}
Let $S=\det^{-1}(s)\sub A(E,D)$ be a nonempty fiber of the map $\det$, where $s\neq 0$.
Then $S\sub A(E,D)$ (resp., $\PP S\sub M^\reg(E,D)$) is a Poisson subvariety and the union of a finite number of symplectic leaves. The natural map
$\pi:S\to \PP S$, which is an \'etale covering of degree $r$, is Poisson. For each symplectic leaf $F\sub \PP S$,
each connected component of the preimage $\pi^{-1}(F)$ is a symplectic leaf in $A(E,D)$.
If $s\in H^0(E,\cO(rD))$ has only simple zeros then $S$ is a smooth symplectic subvariety in $A(E,D)$.
\end{corollary}

\begin{proof}
The first assertion follows from the fact that $\det$ and its projectivization are Casimir maps.
For $\phi\in \PP S$ the map $\phi:E\to E(D)$ degenerates exactly on zeros of $s$. Therefore,
$T=\coker(\phi)$ is a torsion sheaf supported on the divisor of zeros of $s$. There are only finitely many
possibilities for $T$, so $\PP S$ is a union of finitely many symplectic leaves.
The fact that the projection $\pi:S\to \PP S$ is Poisson follows from the same fact about the morphism 
$A(E,D)\setminus \{0\}\to M(E,D)$. If $F\sub \PP S$ is a symplectic leaf then the map $\pi^{-1}(F)\to F$ is \'etale Poisson,
hence, $\pi^{-1}(F)$ is symplectic. 

In the case when the divisor of zeros, $\divis(s)$, is simple, the fact that degree of $T$ is equal to the degree of 
$\divis(s)$
implies that $T\simeq \cO_{\divis(s)}$. Thus, in this case $\PP S$ is a single symplectic leaf, which implies that
$S$ is symplectic. 
\end{proof}

\subsection{Stratification of $M^\reg$}\label{sec:deformleaf}
In Section \ref{sec:classifyleaf}, we gave a rough classification for symplectic leaves of $M^{\reg}$ by showing which torsion sheaves can occur as a cokernel. Now we are going to show that $M^{\reg}$ can be stratified by locally closed subschemes, where each stratum has a structure of a smooth fibration with symplectic leaves as fibers. 

Let $\TS(C)$ be the substack of $\coh(C)$ consisting of torsion sheaves. The connected components of $\TS(C)$ is indexed by the length of the sheaf
\[
\TS(C)=\bigsqcup_{l\geq 1} \TS^l(C).
\] 
For $T\in\TS^l(C)$, its cycle class $[T]$ lies in the $l$-th symmetric product $S^lC$.
\begin{lemma}
The assignment $T\mapsto [T]$ defines a morphism of stacks $f: \TS^l(C)\to S^lC$. 
\end{lemma}
\begin{proof}
This is a special case of corollary 7.15 of \cite{Ry08}.
\end{proof}


The assignments
\[
\{\phi:E\to E(D)\} \mapsto \Cok~\phi\mapsto [\Cok~\phi]
\] defines a sequence of morphisms
\[
M^{\reg}\to \TS^{rk}(C) \to S^{rk}C.
\]
Let $AJ:S^{rk}C\to \Pic^{rk}(C)$ denote the Abel-Jacobi map. Using the fact that $\det(\Cok(\phi))\simeq \cO(rD)$
for $\phi$ in $\cM^{\reg}$, we obtain morphisms
\[
M^{\reg}\to \TS^{rD}(C) \to \PP H^0(C,\cO(rD)) \sub S^{rk}C,
\]
where $\TS^{rD}(C)\sub \TS(C)$ is the substack of torsion sheaves $T$ with $\det(T)\simeq\cO(rD)$
(here we identify the preimage of the Abel-Jacobi map over $\cO(rD)$ with $\PP H^0(C,\cO(rD))$).
Note that the composed map
$M^{\reg} \to \PP H^0(C,\cO(rD))$
is precisely the Casimir map in Proposition \ref{detcasimir}.

For a partition $\nu=(\nu_1\geq \nu_2\ldots\geq\nu_n>0)$ of $l$, we define a locally closed subscheme of $S^l(C)$,
\[
S_\nu C:=\Big\{\sum_{i=1}^n \nu_i[x_i]| x_i\neq x_j ~\text{for}~ i\neq j\Big\}.
\]
Note that we have a Galois covering
$$\rho_\nu:C^{n,\dis}\to S_\nu C$$ 
where $C^{n,\dis}\sub C^n$ is the complement to all the diagonals $x_i=x_j$. The corresponding Galois group
acts by permutations of the coordinates $x_i$ with equal parts $\nu_i$.



Now suppose that in addition to
a partition $\nu=(\nu_1\geq\nu_2\ldots\geq \nu_n)$ of $l$ we have
a collection of partitions $\Lambda_\nu:=\{\lambda^i\}_{i=1}^n$ so that $|\lambda^i|=\nu_i$  for all $i=1,\ldots,n$. 
Then we would like to define the stratum $\TS_{\nu,\Lambda_\nu}$ in the stack $\TS^l$ of torsion sheaves
of length $l$, corresponding to the sheaves of the form
\begin{equation}\label{partition-sheaf-eq}
 \bigoplus_{i=1}^n \left(\bigoplus_{j=1}^{l(\lambda^i)} \cO_{\lambda_j^i x_i}\right)\cong \bigoplus_{i=1}^n \left(\bigoplus_{j\geq 1} \cO_{j x_i}^{\oplus m_j(\lambda(x_i))}\right)
\end{equation} 
 for a distinct collection of points $x_1,\ldots,x_n$. Note that the underlying cycle of such a sheaf would lie in the stratum 
$S_\nu C$ of $S^lC$.

First, we observe that we can determine the isomorphism type of a torsion $\cO_{C,x}$-module $T$, i.e., the partition 
$\lambda=(\lambda_1\ge\ldots\ge \lambda_k)$ such that 
$$T\simeq \bigoplus_{j=1}^k \cO_{C,x}/\fm^{\lambda_j},$$
where $\fm\sub\cO_{C,x}$ is the maximal ideal,
from the numbers 
$$a_j:=\ell(T\otimes_{\cO_{C,x}} \cO_{C,x}/\fm^j).$$
Namely, let $m_p$ be the multiplicity of the part $p$ in the partition $\lambda$. Then one has
$$a_j=\sum_{p=1}^j pm_p+ j\cdot \sum_{p>j} m_p \ \text{ for } j\ge 1.$$
These equations can be solved for $m_1,m_2,\ldots$, so the numbers $(a_j)$ determine the partition $\lambda$.

Thus, the substack $\TS_{\nu,\Lambda_\nu}\sub \TS^l$ should parametrize torsion sheaves $T$ with the underlying cycle 
$\nu_1x_1+\ldots+\nu_nx_n$ such that the lengths $\ell(T\otimes \cO_{jx_i})$ are fixed (and computed from the partitions
$\lambda_i$). Let $\TS_\nu\sub \TS^l$ denote the preimage of the stratum $S_\nu C\sub S^l C$.
Given a flat family of torsion sheaves $\cT$ in $\TS_\nu$ over the base $S$ (so $\cT$ is a sheaf on $S\times C$),
we can consider the induced \'etale covering $S'\to S$, corresponding to the \'etale covering $\rho_\nu:C^{n,\dis}\to C_\nu$,
so that we have the support map 
$$(x_1,\ldots,x_n):S'\to C^{n,\dis}.$$
Let $\cT'$ be the pull-back of $\cT$ to $S'\times C$ and let $\sigma_i:S'\to S'\times C$ be the section corresponding to $x_i$.
Let also $J_i\sub \cO_{S'\times C}$ be the ideal sheaf of the image of $\sigma_i$.
Then by definition, the requirement that $\cT$ is in $\TS_{\nu,\Lambda_\nu}\sub \TS^l$ means that the sheaves
$\cT'\otimes \cO/J_i^j$ are flat over $S'$ and their push-forwards to $S'$ have given rank (determined by the partitions 
$\lambda_i$). 

It is easy to see that in fact $\TS_{\nu,\Lambda_\nu}$ is a locally closed substack of $\TS^l$.
Indeed, given a flat family $T$ of sheaves in $\TS^l$ over a scheme $S$, locally over $S$ we can find a surjection
$\cO^{\oplus l}\to T$, so that $T$ will be realized as a quotient of a cokernel of vector bundles $f:V\to V'$.
Then we can use the well known fact that the fibered product of $\TS_{\nu,\Lambda_\nu}$ with $S$ will
be given as intersection of degeneracy and non-degeneracy loci of $f$ (see \cite[Lem.\ 4.1.3]{P-17}).

Note that the morphism $\TS_{\nu,\Lambda_\nu}\to S_{\nu}C$ factors through as the composition
$$\TS_{\nu,\Lambda_\nu}\to S_{\nu,\Lambda_\nu}\to S_{\nu}C,$$
where $S_{\nu,\Lambda_\nu}$ is the quotient of $C^{n,\dis}$ by the subgroup of $\sigma\in S_n$ compatible
with $\Lambda_\nu$, i.e., such that $\la^{\sigma(i)}=\la^i$ for $i=1,\ldots,n$.

Let us denote by $\TS^{rD}_{\nu,\Lambda_\nu}$ the locally closed substack of $\TS^{rD}$ obtained as the intersection
with $\TS_{\nu,\Lambda_\nu}$.

For a collection of partitions $\Lambda_\nu=\{\lambda^i\}^n_{i=1}$, we define 
\[
l_{\max}(\Lambda_\nu):=\max\{l(\lambda^i)|i=1,\ldots n\}.
\]
For each $\Lambda_\nu$ with $l_{\max}(\Lambda_\nu)\le r$ we define the stratum $M_{\nu,\Lambda_\nu}^{\reg}$,
as a locally closed subscheme of $M^{\reg}$ defined as the fiber product
\begin{equation}\label{fiberdig1}
\xymatrix{
M_{\nu,\Lambda_\nu}^{\reg}\ar[r]\ar[d] &M^{\reg}\ar[d]^q \\
\TS^{rD}_{\nu,\Lambda_\nu}\ar[r] & \TS^{rD}(C)
}
\end{equation}
This defines a finite stratification on $M^{\reg}$.


The following theorem is the main result of this section, which says the map $\delta$ has smooth restrictions to the stratification 
$\{M_{\nu,\Lambda_\nu}^{reg}\}_{\nu,\Lambda_\nu}$.

\begin{theorem}\label{fibrationleaf} For a partition $\nu$ of $r\cdot k$ and a collection of partitions $\Lambda_\nu$ as above,
such that $l_{\max}(\Lambda_\nu)\le r$,
the restriction of $\delta$,
\[
\delta_{\nu,\Lambda_\nu}: M_{\nu,\Lambda_\nu}^{reg}\to S^{rD}_\nu C
\] is smooth and surjective.
\end{theorem}

We need the following Lemma for the proof.

\begin{lemma}\label{surj}
Let $\phi:E\to E(D)$ be an injective morphism. Denote $T$ for $\Cok(\phi)$. Given a first order deformation $T_\ep$ of $T$ that preserves $\det(T)$, there exists a first order deformation $\phi_\ep: E\to E(D)$ so that $\Cok(\phi_\ep)\cong T_\ep$.
\end{lemma}
\begin{proof}
Note that we can identify the open subset of $\PP \Hom(E,E(D))$ consisting of injective $\phi$ with the 
locally closed subset of the Quot-scheme of quotients $\psi:E(D)\to T$ such that $\ker(\psi)$ is semistable and
$\det(T)\simeq \cO(rD)$. Thus, it is enough to check that given a coherent sheaf $\cT$  with $0$-dimensional support
on $C\times S$, where $S=\operatorname{Spec} \CC[\ep]/(\ep^2)$, such that $\cT$ is flat over $S$, and
$T=\cT|_s$ (where $s\in S$ is the unique point) is a quotient of $E(D)$ then there exists a surjection
$p_C^*E(D)\to \cT$, where $p_C:C\times S\to C$ is the projection. By the theorem on cohomology and the base change,
the natural restriction map 
$$H^0(C\times S, p_C^*E^\vee(-D)\otimes\cT)\to H^0(C,E^\vee(-D)\otimes T)$$
is surjective, so we can extend a surjection $E(D)\to T$ to a morphism $\psi:p_C^*E(D)\to \cT$. Since $\psi|_s$ is surjective,
$\psi$ is surjective as well (by Nakayama lemma).
\end{proof}

\begin{proof}[Proof of Theorem \ref{fibrationleaf}]

Note that  the surjectivity of $\delta_{\nu,\Lambda_\nu}$ follows from Theorem \ref{lem_classleaf1}.

By Lemma \ref{surj}, the tangent map to $q$ in diagram \eqref{fiberdig1} is surjective.
This implies that the scheme $M_{\nu,\Lambda_\nu}^{\reg}$ is smooth.
Also, the subscheme $S^{rD}_\nu C\sub S_\nu C$ is smooth as the fiber of the smooth morphism
$AJ_\nu: S_\nu C\to \Pic^{rk}(C)$.

Next, we claim that the morphism $\delta_{\nu,\Lambda_\nu}$ induces surjection on tangent spaces.
Indeed, this follows from its factorization into the composition
$$M_{\nu,\Lambda_\nu}^{reg}\to \TS^{rD}_{\nu,\Lambda_\nu}\to S^{rD}_{\nu,\Lambda_\nu} C \to S^{rD}_\nu C$$
where the last arrow is \'etale, the first arrow is surjective by Lemma \ref{surj}, and the middle arrow admits a section,
$$S^{rD}_{\nu,\Lambda_\nu} C\to \TS^{rD}_{\nu,\Lambda_\nu},$$
given by the family of sheaves \eqref{partition-sheaf-eq}.

Thus, $\delta_{\nu,\Lambda_\nu}$ is a surjective morphism between smooth schemes with surjective tangent map. Therefore it is a smooth morphism.
\end{proof}

\subsection{Products of symplectic leaves}\label{sec:multleaf}
Recall that for any pair of divisors $D$ and $D'$ of degrees $k>0$ and $k'>0$, we have Poisson morphisms 
$A(E,D)\times A(E,D')\to A(E,D+D')$ and
\[
\mu: M(E,D)\times M(E,D')\to M(E,D+D'),
\]
induced by the product map (see Proposition \ref{prod-Poi}). It is clear that $\mu$ sends
$M^{\reg}(E,D)\times M^{\reg}(E,D')$ to $M^{\reg}(E,D+D')$.

The following result might find an application in the study of irreducible representations of quantization of $U(\gl[[t]])$.
Recall that for a torsion sheaf $T\in \TS(C)$ we have the corresponding symplectic leaf $F^c_T$
(defined as the coarse moduli space of the homotopy fiber of $T$ under $q$).

\begin{theorem}\label{multleaf}
Fix $T_1 \in \TS^{rD_1}(C)$ and $T_2 \in \TS^{rD_2}(C)$ such that $F^c_{T_1}$ and $F^c_{T_2}$ are non-empty. 
Let us denote by $[T_1,T_2]$ the set of isomorphism classes of
torsion sheaves $T\in \TS^{r(D_1+D_2)}(C)$ such that there exists a short exact sequence
\begin{equation}\label{torsion-sheaves-ex-seq}
\xymatrix{
0\ar[r] & T_2\ar[r] & T\ar[r] & T_1\ar[r] &0
}
\end{equation}
Then 
$\mu(F^c_{T_1}\times F^c_{T_2})$ is contained in $\bigsqcup_{\substack{T\in [T_1,T_2] \\ l_{\max}(T)\leq r}} F^c_T$. 
Moreover, for each $T\in [T_1,T_2]$ the intersection $\mu(F^c_{T_1}\times F^c_{T_2})\cap F_T^c$ is a nonempty open subset
of $F^c_T$.
In the case when the supports of $T_1$ and $T_2$ are disjoint, the set
$[T_1,T_2]$ consists of a single element $T_1\oplus T_2$
and the map 
\begin{equation}\label{mu-direct-sum-map}
F^c_{T_1}\times F^c_{T_2}\rTo{\mu} F^c_{T_1\oplus T_2}
\end{equation}
is an open embedding.
\end{theorem}

\begin{proof}
Let $\phi_1:E\to E(D_1)$ and $\phi_2: E\to E(D_2)$ be two points in $F^c_{T_1}$ and $F^c_{T_2}$. 
Then their composition $\phi_1\circ\phi_2$ is a map from $E$ to $E(D_1+D_2)$. Then $T:=\coker(\phi_1\circ\phi_2)$ 
fits into a short exact sequence of the form
\eqref{torsion-sheaves-ex-seq}, so $T$ is $[T_i,T_j]$.
Also, the condition $l_{\max}(T_{i+j})\leq r$ holds by Theorem \ref{lem_classleaf1}.

To show that for each $T\in [T_1,T_2]$ the intersection $\mu(F^c_{T_1}\times F^c_{T_2})\cap F_T^c$ is nonempty,
let us fix some exact sequence of the form \eqref{torsion-sheaves-ex-seq}, and let $\pi:T\to T_1$ be the surjective map
from this exact sequence. Note that the induced map
$$\Hom(E(D_1+D_2),T)\to \Hom(E(D_1+D_2),T_1):\psi\mapsto \pi\psi$$
is surjective. Now let us define two Zariski open subsets $U_1,U_2\in \Hom(E(D_1+D_2),T)$ as follows: $U_1$ is the set of
surjective $\psi$ such that $\ker(\psi)$ is semistable, and $U_2$ is the set of $\psi$ such that $\pi\psi$ is surjective and
$\ker(\pi\psi)$ is semistable.
Note that $U_1$ is nonempty by Theorem \ref{lem_classleaf1}, and $U_2$ is nonempty by Theorem \ref{lem_classleaf1} and by surjectivity of the map $\psi\mapsto\pi\psi$.
Hence, $U_1\cap U_2$ is nonempty. Now for $\psi\in U_1\cap U_2$, we have $\ker(\psi)\simeq E$ and 
$\ker(\pi\psi)\simeq E(D_2)$. Thus, the composition
$$\phi:E\simeq \ker(\psi)\hookrightarrow \ker(\pi\psi)\hookrightarrow E(D_1+D_2)$$
of the natural embeddings is a point in $\mu(F^c_{T_1}\times F^c_{T_2})\cap F_T^c$.

To show that $\mu(F^c_{T_1}\times F^c_{T_2})\cap F_T^c$ is open in $F_T^c$, let us consider an auxiliary scheme
$\cS$ defined as a locally closed subscheme in $\Hom(T_2,T)\times \Hom(T,T_1)$ consisting of
the maps $\iota:T_2\to T$, $\pi:T\to T_1$ such that $\iota$ is injective, $\pi$ is surjective, and $\pi\circ\iota=0$.
Note that for $(\iota,\pi)\in\cS$ the maps $\iota$ and $\pi$ fit into an exact sequence \eqref{torsion-sheaves-ex-seq}.
Now as above, we have an open subset $\cU\sub \cS\times\Hom(E(D_1+D_2),T)$ consisting of $(\iota,\pi,\psi)$ such
that $\psi$ is surjective, $\ker(\psi)$ is semistable and $\ker(\pi\circ\psi)$ is semistable. Let also
$\cV\sub\Hom(E(D_1+D_2),T))$ be the open subset of $\psi$ such that $\psi$ is surjective and $\ker(\psi)$ is semistable.
As we have seen above, the intersection $\mu(F^c_{T_1}\times F^c_{T_2})\cap F_T^c$ is exactly the image
of the composition 
$$\cU\rTo{p_2} \cV\to F^c_T$$
where the first arrow is induced by the projection $p_2:\cS\times \Hom(E(D_1+D_2),T)\to \Hom(E(D_1+D_2),T)$,
and the second arrow is the natural smooth map corresponding to quotiening by the free action of $\Aut(T)$
(see Proposition \ref{connected-prop}). Thus, both these arrows are flat, and so the image is an open subset of $F^c_T$.

In the case when $T_1$ and $T_2$ have disjoint supports, we have $\Ext^1(T_1,T_2)=0$, so the only element of
$[T_1,T_2]$ is $T_1\oplus T_2$.
We claim that in this case the map \eqref{mu-direct-sum-map}
 induces an isomorphism onto the open subset of $F^c_{T_1\oplus T_2}$
consisting of $\phi:E\to E(D_1+D_2)$ such that the kernel of the composition
$$E(D_1+D_2)\to \coker(\phi)\simeq T_1\oplus T_2\to T_1$$
is semistable. Indeed, it is clear that \eqref{mu-direct-sum-map} factors through this subset. Conversely, if
$\phi$ belongs to this open subset then the above kernel is isomorphic to $E(D_2)$, so we get the required factorization
of $\phi$ into a composition $E\to E(D_2)\to E(D_1+D_2)$.
\end{proof}



\begin{remark}
In general, the open subsets $\mu(F^c_{T_1}\times F^c_{T_2})\cap F^c_T\subset F^c_T$ in the above theorem a proper,
even in the case when $T_1$ and $T_2$ are disjoint.
We will give an example showing this in Section \ref{nonsur-product-sec}.
\end{remark}

\section{Rank two case}\label{sec:rank-2}
In this section, we go through a rank 2 example in details. This example was the first discovered by Sklyanin in his seminal work \cite{Skl82}. We will compute the symplectic leaves using the method developed in previous sections and make a comparison with Sklyanin's original computation. Then we will give a full classification of leaves for the rank 2 case (Theorem \ref{classrk2leaf}).
\subsection{Sklyanin's example}
Let $\wt{C}=\CC/\Gamma$ with $\Gamma:=\ZZ+\ZZ\tau$ and $C=\CC/\left(\frac{1}{2}\Gamma\right)$ be two complex elliptic curves.  And $D$ be the degree 1 divisor corresponding to the neutral element $e\in C$. 
Denote the group of $2$-torsion points of $\wt{C}$ by $\wt{\Gamma}_2:=\left(\frac{1}{2}\Gamma\right)/\Gamma$, which is kernel of the group homomorphism 
\[
\pi_2: \wt{C} \to C.
\]
It factors through a map $q_2: \wt{C}\to \CC/\left(\ZZ+\ZZ\frac{\tau}{2}\right)$. Let $L^\p$ be the degree 1 line bundle on $\CC/\left(\ZZ+\ZZ\frac{\tau}{2}\right)$, whose associated divisor is the origin. There is a natural action of $\wt{\Gamma}_2$ on $q_2^*{L^\p}\otimes H^0(C,q_2^*{L^\p})^\vee$. Its quotient is a vector bundle $E$ on $C$ of rank 2 and degree 1. It is easy to show that $E$ is indecomposable, therefore stable and $\det(E)\cong \cO_{C}(e)$. 
Denote $\wt{D}$ for $\pi_2^{-1}D$, which is equal to $\wt{\Gamma}_2\subset \wt{C}$.

Let $\lambda\in \CC$ be the complex coordinate of $\CC$. Let 
\[
w_1(\lambda)=\rho(k)\frac{1}{\sn(\lambda,k)}, ~~w_2(\lambda)=\rho(k)\frac{\dn(\lambda,k)}{\sn(\lambda,k)}, ~~w_3(\lambda)=\rho(k)\frac{\cn(\lambda,k)}{\sn(\lambda,k)}
\]
be the Jacobi elliptic function with elliptic modulus $k$ associated to $\wt{C}$ and a normalization constant $\rho(k)$ such that $\Res_{0}w_i(\lambda) d\lambda=1$ for $i=1,2,3$. Let
$\{\sigma_a: a=1,2,3\}$ be the $2\times 2$ Pauli matrices. As a convention, we set $w_0=1$ and $\sigma_0=I_2$.  $w_a$ are meromorphic functions on $\wt{C}$ with simple poles at $\wt{\Gamma}_2$.

The vector space $\Hom_{C}(E,E(D))$, identified with the $\wt{\Gamma}_2$-invariants of $H^0(\wt{C},q_2^*L\otimes \cO_{\wt{C}}(\wt{D}))\otimes H^0(\wt{C},q_2^*L)^\vee$, is spanned by $\{I_2, w_a\sigma_a: a=1,2,3\}$. A general meromorphic endomorphism of $E$ with simple pole at $D$ is of the form
\[
\phi(\lambda)=t_0 I_2+\frac{1}{i}\sum_{a=1}^3 t_aw_a(\lambda)\sigma_a, ~~~t_0,t_a\in \CC.
\]
As elements of $H^0(C,\sEnd(E)(D))$, $w_a(\lambda)\sigma_a$ has simple pole at $e\in C$. For our convienence, we identify the index set of $a$ with $\wt{\Gamma}_2$ and set $w_e(\lambda)=1$ and $J_{ae}=1$ for all $a\in\wt{\Gamma}_2$.

The functions $w_a(\lambda)$ satisfies the quadratic relations
\[
w_a^2(\lambda)-w_b^2(\lambda)=J_{ba}, ~~a,b=1,2,3,
\]
where $J_{ba}$ are constants determined by $k$ and $\rho(k)$.  Using these relations, it is easy to compute that 
\begin{align*}
\det(\phi(\lambda))
&=t_0^2+J_{31}t_1^2+J_{32}t_2^2+(t_1^2+t_2^2+t_3^2)w_3^2(\lambda).
\end{align*}
By Proposition \ref{detcasimir}, determinant  defines  a Casimir map from $A(E,D)\cong \AA^4$ equipped with the Poisson structure \eqref{Poi-3}, to $\AA^2$.
Identifying $t_0,\ldots,t_3$ with the coordinate functions on $\AA^4$, the determinant map is
\begin{align}\label{detSkl}
(t_0,\ldots,t_3)\mapsto (t_0^2+J_{31}t_1^2+J_{32}t_2^2,t_1^2+t_2^2+t_3^2).
\end{align}
The common zero of $t_0^2+J_{31}t_1^2+J_{32}t_2^2$ and $t_1^2+t_2^2+t_3^2$, denoted by $Z$, is isomorphic to $C$. 

Now we pass to the projective case.
Consider the map $q: M(E,D)\to \coh(C)\times\coh(C)$ by
\[
T=(E,E(D),\phi)\mapsto (\Ker(\phi),\Cok(\phi)).
\]
Clearly, $\phi$ is injective if and only if $\det(\phi(\lambda))\neq 0$.

Assume that $\det(\phi(\lambda))\neq 0$. Then it vanishes at $p$ and $q$ such that $p+q$ is lienarly equivalent with $2e$. When $p\neq q$,  $\Cok(\phi)$ must be $\cO_p\oplus\cO_q$. If  $p=q$ then $p$ is a point in the 2-torsion subgroup $\Gamma_2\subset C$, and $\Cok(\phi)$ is either $\cO_p\oplus \cO_p$ or $\cO_{2p}$. There are three types of leaves when $\phi$ is injective.
\begin{enumerate}
\item[(1)] For $p\neq q$, $q^{-1}(\cO_p\oplus\cO_{q})$ is the hypersurface defined by $t_0^2+J_{12}t_2^2+J_{13}t_3^2+\mu(t_1^2+t_2^2+t_3^2)$ for a general $\mu$ in $\PP^1$ without 4 points, taking away $Z$;
\item[(2)] For $a\in \Gamma_2$, $q^{-1}(\cO_{2a})$ is the hypersurface defined by $\sum_{b\neq a} J_{ab} t_b^2$, taking away  $Z\cup P_a$ where $P_a$ is the point with $t_b=0$ for $b\neq a$.
\item[(3)] For $a\in\Gamma_2$, $q^{-1}(\cO_a\oplus \cO_a)$ is the projectivization of the subspace spanned by $w_a(\lambda)\sigma_a$, i.e. a 0-dimensional leaf consisting of a single point. 
\end{enumerate}

Assume that $\det(\phi(\lambda))=0$, then $\phi$ factors through a line bundle $Q$
\[
\phi: E\to Q\to E(D).
\]
Because $\mu(E)=1/2$ and $\mu(E(D))=3/2$, by stability $Q$ must be of degree 1. And the kernel $K$ of the map $E\to Q$ has degree 0. Since $\Hom(E,Q)=\CC$ and $\Hom(Q,E(D))=\CC$, $Q$ (therefore $K$) determines $\phi$ up to isomorphisms.
\begin{enumerate}
\item[(4)] For $K\in \Pic^0(C)$, the fiber of $q$ is a 0-dimensional leaf consisting of a single point on $Z$.
\end{enumerate}
This gives the complete list of symplectic leaves of $M(E,D)$.

\subsection{Classification of leaves for rank two case}
Without loss of generality, we assume that $E$ is a stable vector bundle of rank 2 and degree 1 on $C$ and $D$ be an effective divisor of degree $1$. The goal of this subsection is to classify all symplectic leaves of the Poisson manifolds $M(E,nD)$ for any $n\geq 0$.

\begin{theorem}\label{classrk2leaf}
Let $E$ and $D$ be defined as above. For any $\phi:E\to E(nD)$, one of the following must hold:
\begin{enumerate}
\item[$(1)$] $\phi$ is injective, and $\Cok(\phi)$ is a torsion sheaf satisfying $l_{\max}(\Cok(\phi))\leq 2$ and $\det(\Cok(\phi))\cong \cO(2nD)$.
\item[$(2)$] $\Cok(\phi)\cong L\oplus T$ so that $T$ is a torsion sheaf satisfying $l_{\max}(T)\leq 1$, and $L$ is a line bundle of degree $d(L)$ satisfying 
$$2n+1/2-l(T)>d(L)> n+1/2,$$
 where $l(T)$ refers to length of $T$.
\end{enumerate}
Conversely, suppose a torsion sheaf $T$ (possibly zero) and a line bundle $L$ (possibly zero) satisfy one of the following 
\begin{enumerate}
\item[$(1^\p)$] $L$ is the zero sheaf, and $T$ is a torsion sheaf satisfying $l_{\max}(T)\leq 2$ and $\det(T)\cong \cO(2nD)$.
\item[$(2^\p)$] $T$ is a torsion sheaf satisfying $l_{\max}(T)\leq 1$, and $L$ is a line bundle of degree $d(L)$ satisfying 
$$2n+1/2-l(T)>d(L)> n+1/2,$$ 
\end{enumerate}
there exists a morphism $\phi: E\to E(nD)$ such that $\Cok(\phi)\cong T\oplus L$.
\end{theorem}
\begin{proof}
When $\phi$ is injective,  condition $(1)$ clearly holds. 

Suppose $\phi$ is not injective.
Because $\phi\neq 0$, $\Cok(\phi)$ is isomorphic to $L\oplus T$, where $L$ is a line bundle and $T$ is a torsion sheaf (possibly zero). Denote $K$ for the kernel of the surjection $E(nD)\to L\oplus T$. Let $G$ be the kernel of the surjection $E(nD)\to L$. We have a short exact sequence
\[
\xymatrix{
K\ar[r] & G\ar[r] & T
}\]
Both $K$ and $G$ are line bundles and the composition $E\to K\to G$ is nonzero.
 Because $L$ is a quotient sheaf of $E(nD)$, the inequality $d(L)>n+1/2$ holds. From the above exact sequence, we obtain $d(K)=2n+1-l(T)-d(L)$. The other half of the inequality follows from the slope inequality $\mu(K)>\mu(E)$ because $K$ is a quotient of $E$.

Now we prove the reverse direction. When $T$ satisfies condition $(1^\p)$, the claim follows from Theorem \ref{lem_classleaf1}. Suppose condition $(2^\p)$ holds. The stability implies that there exists a surjection $E(nD)\to L$ with kernel $G$. Because $l_{\max}(T)\leq 1$, we may choose a surjection $G\to T$ with kernel $K$. Clearly, the quotient $E(nD)/K$ is isomorphic to $L\oplus T$. Moreover, by inequality $(2^\p)$ there exists a surjection $E\to K$. By composing it with $K\to E(nD)$, we get the desired map $\phi$.
\end{proof}
In the rank 2 case, the kernel of $\phi$ is a line bundle if it is nonzero. Therefore, it is determined by $\Cok~\phi$ via a determinant calculation. 

\subsection{Example of a non-surjective product map}\label{nonsur-product-sec}

As before, let $E$ be a stable vector bundle of rank $2$ and degree $1$, and $D$ be a divisor of degree 1.
We want to construct an example of two symplectic leaves
$F^c_T$ and $F^c_{T'}$ in $\PP \Hom(E,E(D))$,
corresponding to torsion sheaves of length $2$, $T$ and $T'$, with disjoint support,
such that the product $F^c_T\cdot F^c_{T'}$ is a proper subset of $F^c_{T\oplus T'}$,
i.e., the corresponding map \eqref{mu-direct-sum-map} is not surjective.

Set $L=\det E$. It is well known that $E$ fits into an exact sequence 
$$0\to \cO\to E\rTo{p} L\to 0$$
We have the induced exact sequence
$$0\to \Hom(L,L(D))\to \Hom(E,L(D))\to \Hom(\cO,L(D))\to 0$$
Thus, we can choose an element $\alpha\in \Hom(E,L(D))$ such that its restriction to $\cO$
is any nonzero element $s\in H^0(L(D))$. Then the morphism 
$$f=(p,\alpha):E\to L\oplus L(D)$$
is injective. 
Similarly (e.g., using duality) we can construct an injective morphism
$$g:L\oplus L(D)\to E(2D).$$
Futhermore, the cokernel of $f$ is isomorphic to $\cO_{Z(s)}$, where $Z(s)$ is the divisor of 
zeros of the section $s\in H^0(L(D))$, while $\coker(g)$ is of the form $\cO_{Z(s')}$, where $s'$ is a global section of
$L^{-1}(3D)$. Thus, we can choose $f$ and $g$
in such a way that $T=\coker(f)$ and $T'=\coker(g)$ have disjoint support

Now the composition $\phi:=g\circ f$ gives an injection $E\to E(2D)$ and its cokernel fits into an exact sequence
$$0\to \coker(f)\to \coker(\phi)\to \coker(g)\to 0$$
Now we claim that $\phi\in \Hom(E,E(2D))$ cannot be decomposed as a composition $E\rTo{f'} E(D)\rTo{g'} E(2D)$,
with $\coker(f')\simeq T$ and $\coker(g')\simeq T'$. 
Indeed, otherwise in the corresponding exact sequence
$$0\to \coker(f')\to \coker(\phi)\to \coker(g')\to 0$$
the map $T\oplus T'\simeq \coker(\phi)\to \coker(g')\simeq T'$
would differ from the standard projection $T\oplus T'\to T'$ by an automorphism of $T'$,
so we would get
$$E(D)\simeq \ker(E(2D)\to \coker(g'))\simeq \ker(E(2D)\to \coker(g))\simeq L\oplus L(D),$$
which is a contradiction.

\section{Relation to Feigin-Odesskii Poisson structures}\label{sec:Rel-FO}

In \cite{FO95}, Feigin and Odesskii constructed quadratic Poisson structures on $\PP\Ext^1(\xi,\cO_C)$ for any stable vector bundle $\xi$.
In this section, we show that there exists a Poisson isomorphism between $M(E,D)$ and $\PP\Ext^1(\xi,\cO_C)$, with an appropriate choice of $\xi$. Recall that the Poisson structure on the space $\PP\Ext^1(F_0,F_1)$, where $F_0$ and $F_1$
are stable vector bundles, is obtained by identifying this space with the moduli space of
complexes of the form $[F_0\rTo{\phi} F]$ with $\ker(\phi)=0$ and $\coker(\phi)\simeq F_1$ (see \cite[Sec.\ 5]{HP17}
and Theorem \ref{HPmainthm}).
The Poisson bracket of Feigin-Odesskii coincides with the classical shadow of the 0-shifted Poisson structure of Theorem
\ref{HPmainthm} in the case of $F_0=\cO_C$ and $F_1\cong \xi$. 
We adopt the notation of \cite{FO95} to denote this Poisson manifold by $q_{d,r}(\xi)$,
where $\xi$ is a stable vector bundle of rank $r$ and degree $d$.

The following result is a generalization of \cite[Prop.\ 4.1]{Pol98} (we also fill in some details omitted in \cite{Pol98}).

\begin{theorem}\label{FM-thm}
Let $\Phi$ be an autoequivalence of $\D^b(\coh(C))$, which, up to a shift, is isomorphic to the composition of spherical twists
and their inverses.
Let $E_0$ and $E_1$ be a pair of stable vector bundles such that $\mu(E_0)<\mu(E_1)$ and $\rk(E_0)\le \rk(E_1)$.
Assume that $\Phi(E_0)=F_0$, $\Phi(E_1)=F_1[1]$ for some stable bundles $F_0$ and $F_1$.
Then the isomorphism
$$\PP \Hom(E_0,E_1)\rTo{\sim} \PP \Ext^1(F_0,F_1)$$
induced by $\Phi$ is compatible with the Poisson structures.
\end{theorem}

Before proving the theorem, we point out the corollary relevant for our setup.

\begin{corollary}\label{FM-cor}
Fix a stable vector bundle $E$ of rank $r$ and degree $d$ so that $0<d<r$ and $\gcd(r,d)=1$, and an effective divisor $D$ of degree $k$. Let $m,n$ be the pair of integers so that $mr+nd=1$ and $n$ is the smallest positive integer satisfying this. There exists a stable vector bundle $\xi$ of rank $rkn-1$ and degree $r^2k$ so that $M(E,D)$ is isomorphic to $q_{r^2k,rkn-1}(\xi)$ as Poisson manifolds.
\end{corollary}

\begin{proof}
By Theorem \ref{FM-thm},
it is enough to construct an autoequivalence $\Phi$ of $\D^b(\coh(C))$, which is a composition of spherical twists, such that 
\[
\Phi(E)\simeq\xi,~~~\Phi(E(D))\simeq\cO_C[1].
\]
To see that such an autoequivalence exists, first, let us consider this problem on the level of discrete invariants.
The vectors $(\deg, \rk)$ for $E$ and $E(D)$ have form $(r,d)$ and $(r,d+kr)$. While these vectors for $\xi$ and $\cO_C[1]$
are $(rkn-1,r^2k)$ and $(-1,0)$. It is well known that the $\SL_2(\ZZ)$-orbit of a pair of vectors $v_1,v_2$ is determined
by two invariants: $\det(v_1,v_2)$ and the invariant $\alpha(v_1,v_2)\in (\ZZ/\det(v_1,v_2)\ZZ)^*$ such that
$$v_1\equiv \alpha(v_1,v_2)v_2\mod \det(v_1,v_2)\cdot\ZZ^2$$
It is easy to check that for both pairs the determinant is $kr^2$, while $\alpha\equiv 1-rkn$. 
Thus, there exists an element of $\SL_2(\ZZ)$ sending the pair $(r,d), (r,d+kr)$ to $(rkn-1,r^2k), (-1,0)$.
We can realize it by a composition of spherical twists, $\Phi$, so that $\Phi(E(D))\simeq \cO_C[1]$. Then $\Phi(E)$ is necessarily
a vector bundle of the required form.
\end{proof}

\begin{lemma}\label{FM-lem}
Let $\Phi$ be an autoequivalence of $\D^b(\coh(C))$, which is a composition of spherical relections, their inverses, and of a shift.
Then $\Phi$ is compatible with the 1-Calabi-Yau structure on $\D^b(\coh(C))$ 
(coming from a fixed trivialization $\omega_C\simeq \cO_C$):
the composition
$$\Hom(\Phi(A),\Phi(B))^*\simeq\Hom(A,B)^*\rTo{SD(A,B)} \Hom(B,A[1])\simeq \Hom(\Phi(B),\Phi(A)[1])$$
coincides with the Serre duality isomorphism $SD(\Phi(A),\Phi(B))$.
\end{lemma}

\begin{proof} It is easy to see that this assertion is equivalent to the fact that $\Phi$ acts trivially on 
$HH_1(C)=\Hom(\cS^{-1},\Id[-1])$, where $\cS^{-1}$ is the inverse of the Serre functor.
It is enough to consider the case when $\Phi$ is a spherical twist with respect to a spherical object $E$.
Then the action of $\Phi$ on $HH_*(C)$ is given by the formula
$$x\mapsto x-\lan \ch(E),x\ran\cdot \ch(E),$$
where $\ch(E)$ is the Chern character of $E$ with values in $HH_*(C)$ and $\lan\cdot,\cdot\ran$ is the canonical
pairing (the proof is similar to \cite[Lem.\ 8.12]{Huy}).
Since $\ch(E)\in HH_0(C)$ and $HH_0(C)$ is orthogonal to $HH_1(C)$, we deduce that $\Phi$ acts trivially on
$HH_1(C)$.
\end{proof}

Let us say that a coherent sheaf on $C$ is {\it semistable} if it is either a semistable bundle or a torsion sheaf. 
By semistable objects in $\D^b(\coh(C))$ we mean objects of the form $F[n]$, where $F$ is a semistable coherent sheaf.
Such an object is called {\it stable} if its degree and rank are coprime (for bundles this is equivalent to the usual stability).

\begin{lemma}\label{semistable-bun-lem}
Let $A$ and $B$ be a pair of stable objects in $\D^b(\coh(C))$ of distinct slopes, such that $\Hom(A,B)\neq 0$.
Then for a generic morphism $f:A\to B$ the cone $\Cone(f)$ is semistable.
\end{lemma}

\begin{proof} The set of $f\in\Hom(A,B)$ such that $\Cone(f)$ is semistable is open, so it is enough to find one such $f$.
Applying an autoequivalence of $\D^b(\coh(C))$, we can assume that $B=\cO_p$ for some point $p\in C$.
Since $\Hom(A,\cO_p)\neq 0$, this implies that $A$ is a stable vector bundle. Thus, we can find an isogeny of elliptic curves
$f:C'\to C$ and a line bundle $L$ on $C'$ such that $A\simeq f_*L$. Let $q\in C'$ be a point such that $f(q)=p$.
Then the push-forward of the exact sequence
$$0\to L(-q)\to L\to L|_q\to 0$$
gives an exact sequence
$$0\to f_*(L(-q))\to A\to B\to 0$$
It remains to use the fact that $f_*(L(-q))$ is semistable (see Proposition \ref{lem_pushfor}).
\end{proof}

\begin{proof}[Proof of Theorem \ref{FM-thm}]
It suffices to check that our isomorphism is compatible with the Poisson structures
over some dense open subsets. So we will only consider the open subset 
$$U\sub \PP \Hom(E_0,E_1)$$ 
corresponding to injective $\phi:E_0\to E_1$, such that $\coker(\phi)$
is a semistable sheaf (in the case $\rk(E_0)=\rk(E_1)$ the cokernel is a torsion sheaf, so the semistability is automatic). 
Lemma \ref{semistable-bun-lem} easily implies that the set $U$ is nonempty. Indeed, we have
$$\Cone(\phi)\simeq \ker(\phi)[1]\oplus\coker(\phi),$$ 
so the semistability of $\Cone(\phi)$ implies that $\phi$ is either injective or surjective. 
However, since $\rk(E_0)\le \rk(E_1)$, if $\phi$ is surjective, it would be an isomorphism,
which contradicts the assumption $\mu(E_0)<\mu(E_1)$.

Let us also consider the open subset 
$$V\subset \PP \Ext^1(F_0,F_1)$$ 
consisting of extension classes such that the corresponding extended bundle $F$ is 
semistable. We claim that the map given by $\Phi$ restricts to an isomorphism $U\to V$.
Indeed, in the case when $\phi$ is injective, applying $\Phi$ to the exact triangle
$$E_0\rTo{\phi} E_1\to \coker(\phi)\to E_0[1]$$
we should get the exact triangle
$$F_0\to F_1[1]\to F[1]\to F_0[1],$$
so $F[1]\simeq \Phi(\coker(\phi))$, which shows that $F$ is semistable when $\coker(\phi)$ is.
Conversely, assume $F$ is semistable. Then $\Phi^{-1}(F[1])$ is semistable, so $\phi$ is injective and $\coker(\phi)$ is semistable.

Now let us fix an element $\lan\phi_0\ran$ in $U$, let $E=\coker(\phi_0)$, so that we have an exact sequence
\begin{equation}\label{E-T-ex-seq}
0\to E_0\rTo{\phi_0} E_1\to E\to 0
\end{equation}
and let $F=\Phi(E)[-1]$ be the corresponding
semistable extension of $F_0$ by $F_1$.
Let us consider two auxiliary spaces depending on $\phi_0$. First, let $U'$ be the open subset in $\PP \Hom(E_1,E)$
consisting of surjective $\psi:E_1\to E$ such that $\ker(\psi)$ is stable (and hence,
isomorphic to $E_0$). The other space $V'$ is defined as the open subset in $\PP \Hom(F_1,F)$ corresponding
to $f:F_1\to F$ such that $\coker(f)$ is a stable bundle (and hence, isomorphic to $F_0$).
Note that we have a commutative diagram, where the vertical arrows are induced by $\Phi$:
\begin{diagram}
U'&\rTo{\rho}&U\\
\dTo{\Phi}&&\dTo{\Phi}\\
V'&\rTo{\tau}&V
\end{diagram}
Here the map $\rho:U'\to U$ associates to $\psi:E_1\to E$ the corresponding morphism $E_0\simeq\ker(\psi)\to E_1$,
and the map $\tau:V'\to V$ associates to $f:F_1\to F$ the class of the extension of $\coker(f)\simeq F_0$ by $F_1$.
Let $\psi_0:E_1\to E$ be the canonical projection, and let $f_0:F_1\to F$ be the image of $\psi_0$ under $\Phi[-1]$.
Then the above diagram is compatible with distinguished points in all spaces: $\lan\psi_0\ran\in U'$, $\lan f_0\ran\in V'$,
$\lan\phi_0\ran\in U$, and the class $e_0$ of the extension $F$ in $V$.

Now we claim that there are natural isomorphisms
$$T^*_{\phi_0}U\rTo{\alpha} T_{\psi_0}U', \ \ T^*_{e_0}V\rTo{\beta}T_{f_0}V',$$
fitting into the commutative diagram 
\begin{equation}\label{FM-squares-diagram}
\begin{diagram}
T^*_{\phi_0}U&\rTo{\alpha}&T_{\psi_0}U'&\rTo{d\rho}&T_{\phi_0}U\\
\uTo{d\Phi^*}&&\dTo{d\Phi}&&\dTo{d\Phi}\\
T^*_{e_0}V&\rTo{\beta}&T_{f_0}V'&\rTo{d\tau}&T_{e_0}V
\end{diagram}
\end{equation}
in which the compositions given by the two rows are the Poisson tensors $\Pi^U$ and $\Pi^V$,
computed at $\phi_0$ and $e_0$,
respectively.

To define $\alpha$, we note that by Serre duality, we have an identification
$$T^*_{\phi_0}U\simeq \ker(\Ext^1(E_1,E_0)\to \Ext^1(E_1,E_1)).$$
Now the long exact sequence of the functors $\Ext^*(E_1,?)$ applied to the exact sequence \eqref{E-T-ex-seq}
gives an isomorphism of the latter kernel with 
$$\coker(\Hom(E_1,E_1)\to \Hom(E_1,E))\simeq T_{\psi_0}U'.$$
Similarly, to define $\beta$, we use Serre duality to get an identification
$$T^*_{e_0}V\simeq \ker(\Hom(F_1,F_0)\to \Ext^1(F_1,F_1)),$$
and use the isomorphism of the latter kernel with  
$$\coker(\Hom(F_1,F_1)\to \Hom(F_1,F))\simeq T_{f_0}V'.$$

Note that the commutativity of the left square in diagram \eqref{FM-squares-diagram} follows from the fact
that $\Phi$ is compatible with Serre duality (see Lemma \ref{FM-lem}).

On the other hand, the Poisson tensor $\Pi^U$ is induced by the map
$$\Ext^1(E_1,E_0)\simeq H^1(\sHom(E_1,E_0))\to \HH^1(\sEnd(E_1)\oplus \sEnd(E_0)\to \sHom(E_0,E_1)),$$
coming from the chain map
\begin{diagram}
\sHom(E_1,E_0)\\
\dTo{(\phi_0\circ , \circ\phi_0)}\\
\sEnd(E_1)\oplus \sEnd(E_0)&\rTo{\partial}& \sHom(E_0,E_1)
\end{diagram}
Now our claim that $\Pi^U\circ\alpha^{-1}=d\rho$ follows from Lemma \ref{tangent-map-comp-lem} below,
together with commutativity of the diagram
\begin{diagram}
T^*_{\phi_0}U&\rTo{\alpha}& T_{\psi_0}U'\\
\dTo{}&&\dTo{}\\
\Ext^1(E_1,E_0)&\rTo{\sim}&\HH^1(\sEnd(E_1)\to \sHom(E_1,E))
\end{diagram}
where the bottom arrow is induced by the quasiisomorphism
$\sHom(E_1,E_0)\to [\sEnd(E_1)\to \sHom(E_1,E)]$, and the right vertical arrow is induced by the embedding of $U'$
into the space of triples $E_1\to E$ with fixed $E$.

Next, we need to check that $\Pi^V=d\tau\circ\beta$. Recall (see the discussion after Lemma 3.1 in \cite{Pol98})
that the tangent space to the moduli stack $\cM'$ of triples
$(\wt{F}_1\rTo{f} \wt{F})$, where $f$ is an embedding of a subbundle, at a given triple $(F_1\rTo{f_0}F)$, can be identified with
$H^1(\sEnd(F,F_1))$, where $\sEnd(F,F_1)$ is the bundle of endomorphisms of $F$ preserving the subbundle $F_1$.
Furthermore, the Poisson tensor $\Pi^V$ is given by the composition
$$H^1(\sEnd(F,F_1))^*\rTo{\sim} H^0(\sEnd(F,F_1)^\vee)\rTo{\gamma} H^0(\sHom(F_1,F_0))\rTo{\delta} H^1(\sEnd(F,F_1)),$$
where the first isomorphism is given by the Serre duality; $\gamma$ comes the natural morphism
$\sEnd(F,F_1)^\vee\to\sHom(F_1,F_0)$ (which is dual to the embedding $\sHom(F_0,F_1)\to\sEnd(F,F_1)$);
and $\delta$ is the boundary homomorphism coming from the exact sequence
$$0\to \sEnd(F,F_1)\to \sEnd(F)\to \sHom(F_1,F_0)\to 0$$
Note that the tangent map to the natural embedding 
$$\PP \Ext^1(F_0,F_1)\rTo{i} \cM'$$
is the map 
$$di: H^1(\sHom(F_0,F_1))/\lan e_0\ran\to H^1(\sEnd(F,F_1))$$
induced by the embedding of bundles $\sHom(F_0,F_1)\to \sEnd(F,F_1)$.
By Serre duality, the dual projection is the map
$$di^*:H^0(\sEnd(F,F_1)^\vee)\to \ker(H^0(\sHom(F_1,F_0))\to\Ext^1(F_1,F_1))\simeq T^*_{e_0}V$$
Thus, it is enough to check the commutativity of the following diagram
\begin{diagram}
H^0(\sEnd(F,F_1)^\vee)&\rTo{\gamma}&\Hom(F_1,F_0)&\rTo{\delta}&H^1(\sEnd(F,F_1))\\
\dTo{di^*}&&\uTo{}&&\uTo{di}\\
T^*_{e_0}V&\rTo{\beta}&T_{f_0}V'&\rTo{d\tau}&T_{e_0}V
\end{diagram}
where the middle vertical arrow is induced by the projection $F\to F_0$.
The commutativity of the left square follows easily from the definitions of the relevant maps.
To prove the commutativity of the right square we observe that $di\circ d\tau$ is just the tangent map
to the natural embedding $V'\hookrightarrow \cM'$. Thus, it remains to check that the tangent map to this embedding
is given by the composition
$$\Hom(F_1,F)/\Hom(F_1,F_1)\to \Hom(F_1,F_0)\rTo{\delta} H^1(\sEnd(F,F_1)).$$
Indeed, under the identification of the tangent space to $\cM'$ with 
$\HH^1(\sEnd(F)\oplus\sEnd(F_1)\rTo{\partial}\sHom(F_1,F))$ this tangent map is induced by the natural map
$$H^0(\sHom(F_1,F))\to \HH^1(\sEnd(F)\oplus\sEnd(F_1)\to\sHom(F_1,F)).$$
Similarly, $\delta$ is induced by the natural map
$$H^0(\sHom(F_1,F_0))\to \HH^1(\sEnd(F)\to \sHom(F_1,F_0))\simeq H^1(\sEnd(F,F_1)).$$
Now our statement follows from the commutativity of the square
\begin{diagram}
\sEnd(F)\oplus\sEnd(F_1)&\rTo{\partial}&\sHom(F_1,F)\\
\dTo{}&&\dTo{}\\
\sEnd(F)&\rTo{}&\sHom(F_1,F_0)
\end{diagram}
\end{proof}

\begin{lemma}\label{tangent-map-comp-lem} 
Assume that we have an exact sequence of coherent sheaves on a scheme $X$,
$$0\to A_0\rTo{a_0} B_0\rTo{b_0} C_0\to 0,$$
where $A_0$ and $B_0$ are vector bundles.
Let $\cP$ denote the moduli stack of pairs $(B, b:B\to C_0)$ where $B$ is a vector bundle and $b$ is surjective (and $C_0$ is
fixed),
and let $\Trp$ be the moduli stack of triples $(A,B,a:A\to B)$, where $A$ and $B$ are vector bundles.
Let $\rho:\cP\to\Trp$ be the natural morphism sending $(B,b)$ to $(A=\ker(b),B, a:\ker(b)\to B)$, where
$a$ is the natural embedding. Then the tangent map to $\rho$ at the point $(B_0,b_0)$ is the map
on $\HH^1$ induced by the chain map of complexes
\begin{diagram}
\sHom(B_0,A_0)\\
\dTo{\circ a_0,a_0\circ}\\
\sEnd(A_0)\oplus \sEnd(B_0)&\rTo{\partial}&\sHom(A_0,B_0)
\end{diagram}
where we use the quasiisomorphism 
\begin{equation}\label{quis-Hom-BA-eq}
\sHom(B_0,A_0)\to [\sEnd(B_0)\to \sHom(B_0,C_0)] 
\end{equation}
to identify the tangent space to $\cP$ with $H^1(\sHom(B_0,A_0))$.
\end{lemma}

\begin{proof}
Let $\cM$ be the moduli stack of exact complexes of the form $0\to A\to B\to C_0\to 0$ where $A$ and $B$ are vector bundles
and $C_0$ is fixed. Then we have obvious projections $p_{AB}:\cM\to \Trp$ and $p_{BC_0}:\cM\to\cP$, such that
$p_{BC_0}$ is an isomorphism and $\rho\circ p_{BC_0}=p_{AB}$. The tangent space to $\cM$ at 
$(A_0\rTo{a_0}B_0\rTo{b_0}C_0)$ is given by 
$\HH^1(\cE^\bullet)$ with
$$\cE^\bullet=[\sEnd(A_0)\oplus \sEnd(B_0)\to\sHom(A_0,B_0)\oplus \sHom(B_0,C_0)\to\sHom(A_0,C_0)],$$
so that the tangent maps $dp_{AB}$ and $dp_{BC_0}$ are induced by the natural projections of complexes.
Now we observe that there is a natural quasiisomorphism $\sHom(B_0,A_0)\to \cE^\bullet$ lifting the quasiisomorphism
\eqref{quis-Hom-BA-eq}. Thus, the required tangent map is induced by the composition
$$\sHom(B_0,A_0)\to \cE^\bullet\to [\sEnd(A_0)\oplus \sEnd(B_0)\to \sHom(A_0,B_0)].$$
\end{proof}

\begin{remark} One can construct the autoequivalence $\Phi$ from the proof of Corollary \ref{FM-cor} explicitly.
First, let us consider the special case when $d=1$. Then we have $n=1$ and $m=0$.
We adopt the notation of Feigin and Odesskii to denote  a stable vector bundle of rank $r$, degree $d$ and determinant $\alpha$ by $\xi_{d,r}(\alpha)$.
The standard Fourier-Mukai transform $F$ on $\D^b(\coh(C))$ (which is a composition of spherical twists) satisfies 
\begin{enumerate}
\item[$\bullet$] $F(\xi_{d,r}(\alpha))=\xi_{-r,d}(\alpha^{-1})$ if $d>0$ and $F(\xi_{d,r}(\alpha))=\xi_{r,-d}(\alpha^{-1})[-1]$ if $d<0$.
\end{enumerate}
Applying this to $E$, we get that $F(E)=L_r$ is a line bundle of degree $-r$. Define $\Phi$ to be the composition of functors
\[
[1]\circ(-\otimes L_r^{-1})\circ F\circ (-\otimes \cO(-D)).
\]
Clearly, $\Phi(E(D))=\cO_C[1]$. Since $E(-D)$ has rank $r$ and degree $1-rk$, $F(E(-D))$ is a stable bundle of rank $rk-1$ and degree $r$ shifted by $[-1]$. So $\Phi(E)$ is a stable vector bundle of rank $rk-1$ and degree $r^2k$, as required.

In the general case we decompose $d/r$ into a continuous fraction 
\[
\frac{d}{r}=\frac{1}{r_1-\frac{1}{r_2-\ldots\frac{1}{r_p}}}, ~r_i\geq 2,~ 1\leq i\leq p.
\]
There exists line bundles $L_{r_i}$ of degree $-r_i$, for $i=1,\ldots,p$, so that 
 the autoequivalence $\Phi_{r,d}$ defined by
\[
\Phi_{r,d}:= [1]\circ(\otimes L^{-1}_{r_p})\circ F\circ\ldots (\otimes L^{-1}_{r_2})\circ F\circ(\otimes L^{-1}_{r_1})\circ(\otimes\cO(-D)).
\]
satisfies $\Phi_{r,d}(E(D))=\cO_C[1]$. It is easy to check that $\Phi_{r,d}(E)$ is a vector bundle with
\[
\rk(\Phi_{r,d}(E))=r^2k,~~~\deg(\Phi_{r,d}(E))=rkn-1.
\]
\end{remark}

\begin{remark}
It is expected that $q_{r^2k,rk-1}(\xi)$ is quantized by the Feigin-Odesskii elliptic algebra $Q_{r^2k,rk-1}(\eta)$ with $\eta\in C$ (see definition in \cite{FO87}). This is stated in \cite{FO95} without proof. A proof for a special case is given in Section 5 of \cite{HP17} when $\xi$ is assumed to be a line bundle. An important application of Corollary \ref{FM-cor} is that it can be used to construct a comultiplication morphism between different families of Feigin-Odesskii elliptic algebras. The existence of such a comultiplication morphism was first predicted by Feigin and Odesskii in their seminal paper \cite{FO87}. However, Feigin and Odesskii did not give the formula of the comultiplication. Elsewhere we plan to provide an explicit formula for the comultiplication map 
\[
Q_{r^2(k_1+k_2),r(k_1+k_2)-1}(\eta)\to Q_{r^2k_1,rk_1-1}(\eta_1)\otimes Q_{r^2k_2,rk_2-1}(\eta_2).
\]
\end{remark}


\end{document}